\documentclass[11pt]{amsart}
\title{Localization and Gluing of Orbifold Amplitudes: The Gromov-Witten Orbifold Vertex}
\author{Dustin Ross}
\address{Dustin Ross, Colorado State University, Department of Mathematics, Fort Collins, CO 80523-1874, USA}
\email{ross@math.colostate.edu}

\usepackage{graphicx,psfrag,verbatim,amssymb,amscd,enumerate,subfigure}
\usepackage[dvips]{color}
\usepackage{texdraw,xypic,appendix}

\newcommand{\ord}{\text{ord}}

\newcommand{\proj}{\mathbb{P}}
\newcommand{\R}{\mathbb{R}}
\newcommand{\N}{\mathbb{N}}
\newcommand{\Z}{\mathbb{Z}}
\newcommand{\aut}{\mathrm{Aut}}

\newcommand{\M}{\overline{\mathcal{M}}}
\newcommand{\so}{\mathcal{O}}
\newcommand{\Q}{\mathbb{Q}}
\newcommand{\C}{\mathbb{C}}
\newcommand{\X}{\mathcal{X}}

\newcommand{\B}{\mathcal{B}}
\newcommand{\La}{\mathcal{L}}
\newcommand{\Y}{\mathcal{Y}}
\newcommand{\F}{\mathcal{F}}
\newcommand{\fc}{1\hspace{-.1cm}\text{I}}
\newcommand{\G}{\mathcal{G}}

\newtheorem{dummy}{}[section]
\newtheorem{lemma}[dummy]{Lemma}

\newtheorem{theorem}[dummy]{Theorem}

\newtheorem{conjecture}[dummy]{Conjecture}

\theoremstyle{definition}

\newtheorem{definition}[dummy]{Definition}

\newtheorem{remark}[dummy]{Remark}

2
\begin{document}

\subjclass[2010]{Primary 14n35. Secondary  05a15.}

\begin{abstract}
We define a formalism for computing open orbifold GW invariants of $[\C^3/G]$ where $G$ is any finite abelian group.  We prove that this formalism and a suitable gluing algorithm can be used to compute GW invariants in all genera of any toric CY orbifold of dimension $3$.  We conjecture a correspondence with the DT orbifold vertex of Bryan-Cadman-Young.
\end{abstract}

\maketitle
\tableofcontents

\pagebreak

\section{Introduction}\label{sec:intro}

\subsection{Summary of Results}

In \cite{akmv:tv}, Aganagic, Klemm, Mari\~no and Vafa proposed an algorithm for computing Gromov-Witten invariants in all genera of smooth toric Calabi-Yau 3-folds.  The basic building block for their algorithm is the topological vertex, a generating function of open GW invariants which they predict via large N duality.  Recent work has proven the validity of the topological vertex and the corresponding algorithm (\cite{lllz:mttv},\cite{mnop:gwdt1},\cite{orv:qcycc},\cite{moop:gwdtc}).  

This work develops an analogous algorithm for computing GW invariants in all genera of toric CY \textit{orbifolds} of dimension 3.  We now paraphrase the results herein.\\



\noindent \textbf{Definition \ref{def:vertex}:} We define the \textit{GW Orbifold Topological Vertex} via open GW invariants of $[\C^3/G]$ where $G$ is any finite abelian group.\\

The open orbifold GW invariants in Definition \ref{def:vertex} are defined via localization, generalizing the methods of Katz and Liu (\cite{kl:oinv}).  The technical heart of our work is the proof that the GW orbifold vertex glues.\\

\noindent\textbf{Theorem \ref{gluing}:} The GW theory of toric CY orbifolds of dimension 3 is determined by the GW orbifold vertex and a suitable gluing algorithm.\\

Given a planar trivalent graph $\Gamma=\{\text{vertices}, \text{edges}\}$ associated to the target orbifold, the gluing  algorithm has the following form:
\[
Z=\sum_{\Lambda} \prod_{\text{edges}} E(e,\Lambda) \prod_{\text{vertices}} V(v,\Lambda)
\]
where $Z$ is the GW potential of the target, $V(v,\Lambda)$ is the GW orbifold vertex and $E(e,\Lambda)$ consists simply of an automorphism correcting combinatorial factor and a sign which depends on the geometry of the target near $e$.  The sum is over all possible ways of assigning decorated partitions to the edges (Section \ref{sec:edge}).


We conjecture that the GW orbifold vertex developed herein is equivalent to the DT orbifold vertex developed in \cite{bcy:otv} after a change of variables and a change of basis in an associated Fock space.  To provide justification for the conjecture, we turn to the example of $[\C^3/\Z_2]$.  In Section \ref{sec:a1all}, we compute explicit expressions for the one-leg vertex along either the gerby or a non-gerby leg.  In Section \ref{ssec:a1ef} we prove the following.\\


\noindent\textbf{Theorem \ref{effectivecor}:} We formulate and prove the GW/DT vertex correspondence for the effective, one-leg $\Z_2$ vertex.\\

The ineffective case turns out to be both more interesting and more challenging due to the nontrivial isotropy.  We prove the following in Section \ref{ssec:a1in}.\\

\noindent\textbf{Conjecture \ref{conjgwdtin}/Theorem \ref{gwdtin}} We formulate the GW/DT vertex correspondence for the ineffective, one-leg $\Z_2$ vertex and show that it holds in winding 1, in winding 2 with non-negative Euler characteristic, and in winding 3 with positive Euler characteristic.\\

These checks provide a highly nontrivial justification for the correspondence.  In order to prove Theorem \ref{gwdtin}, a new identity for higher genus $\Z_2$-Hodge integrals is required.  Utilizing localization techniques, we prove the identity in Appendix \ref{sec:hhi}.\\

\noindent\textbf{Theorem \ref{hhi}:}
\[
\sum_{i,j,g,n}\int_{M_{g;n,0(\B\Z_2)}}\hspace{-1cm}\lambda_{g-i}\hat\lambda_{\hat g}\hat\lambda_{\hat g-j}\psi^{i+j-1}\frac{x^{n-1}}{(n-1)!}\lambda^{2g-1}=\frac{1}{2}\csc\left(\frac{\lambda}{2}\right)\tan\left(\frac{x}{2}\right)
\]
where $\lambda_i:=c_i(\mathbb{E}_{1})$, $\hat{\lambda}_i:=c_i(\mathbb{E}_{-1})$, $\psi$ is pulled back from $\M_{g,n}$, and $\hat g:=g-1+\frac{n}{2}$.\\

Finally, reverse-engineering Conjecture \ref{conjgwdtin}, we obtain new predictions for certain tautological intersections on the moduli space of twisted curves.\\

\noindent\textbf{Predictions:} Assuming Conjecture \ref{conjgwdtin} in winding $2$, we predict closed forms for generating functions of certain genus 1, 2, and 3 $\Z_2$-Hodge integrals which arise in the $\Z_2$ vertex (Section \ref{sec:pred}).\\

\subsection{Context and Motivation}

GW invariants of a smooth complex manifold $X$ virtually count maps $f:\Sigma\rightarrow X$ where $\Sigma$ is a compact Riemann surface.  In case $X$ is toric, the Atiyah-Bott localization theorem can be used to reduce the computation of GW invariants to the evaluation of an associated graph sum which can, in principle, be computed.  However, the latter computation is often so combinatorially dense that it becomes infeasible to extract any meaningful structure of the invariants.

In \cite{akmv:tv}, Aganagic, Klemm, Mari\~no, and Vafa proposed a method for significantly reducing the combinatorial complexity of the computations when the target is a toric CY 3-fold.  Heuristically, this works as follows.  To each toric CY 3-fold, one can associate a planar trivalent graph, vertices correspond to torus invariant points and edges correspond to torus invariant lines.  Aganagic, Klemm, Mari\~no, and Vafa suggested `cutting' each edge of the graph with a Lagrangian submanifold, leaving a collection of trivalent vertices.  They proposed that the GW invariants of the 3-fold can be parsed into contributions coming from each vertex, and they defined the \textit{topological vertex} to be the contribution of a single trivalent vertex.  They suggested that the GW theory of any toric CY 3-fold can be obtained from the topological vertex and a suitable gluing algorithm.


The `cutting' of the target by Lagrangians naturally suggests the study of \textit{open} GW invariants, i.e. virtual counts of maps from Riemann surfaces with boundary.  Using large N duality from Chern-Simons theory, Aganagic, Klemm, Mari\~no, and Vafa gave explicit predictions for the topological vertex which they claimed to be the open GW theory of $\C^3$ relative to three special Lagrangian submanifolds.  As open GW theory had not yet been mathematically developed, there was no way to verify these predictions at the time.

In \cite{kl:oinv}, Katz and Liu proposed a definition of open GW invariants via localization under certain hypotheses on the target.  In particular, their theory applies to define open GW invariants of toric 3-folds.  An important aspect of the theory developed by Katz and Liu is that different choices of torus action result in different open invariants.  In \cite{df:lgta}, Diaconescu and Florea implemented the methods of Katz and Liu to compute the open invariants of $\C^3$ relative to the Lagrangian submanifolds prescribed by \cite{akmv:tv}.  Diaconescu and Florea then compared their open GW invariants to those predicted by \cite{akmv:tv} and made highly nontrivial checks to justify their theory as a realization of the topological vertex.

Another realization of the topological vertex was developed by Li, Liu, Liu, and Zhou in \cite{lllz:mttv}.  They defined open invariants via (formal) relative GW theory and showed that the resulting vertex recovers the predictions of \cite{akmv:tv} in the one and two-leg cases.  The full three-leg case was left as an open question.  Simply comparing formulas, it is not hard to see that the topological vertex computations of \cite{df:lgta} and \cite{lllz:mttv} are equivalent where the torus action dependency of the former corresponds to the framing dependency of the latter.

In parallel, Okounkov, Reshetikhin, and Vafa defined generating functions for 3d partitions with prescribed asymptotics along the axes (\cite{orv:qcycc}).  They proved that these generating functions are equivalent (up to a normalization factor) to the topological vertex of \cite{akmv:tv}.  In particular, the work of \cite{orv:qcycc} made a connection with Donaldson-Thomas theory -- it was shown via localization in \cite{mnop:gwdt1} that the generating functions of \cite{orv:qcycc} are a natural building block for DT invariants of toric CY 3-folds.  When the GW/DT correspondence for toric 3-folds was proven by Maulik, Oblomkov, Okounkov, and Pandharipande in \cite{moop:gwdtc}, one consequence was the validity of the full three-leg topological vertex of \cite{akmv:tv}.

In recent years, GW and DT invariants have been defined for orbifold targets.  In \cite{bcy:otv}, Bryan, Cadman, and Young defined the DT orbifold topological vertex of $[\C^3/G]$ as the generating function of 3d partitions colored by irreducible representations of $G$ (naturally generalizing the generating functions of \cite{orv:qcycc}).  They proved that the DT orbifold vertex is a basic building block for computing the DT invariants of any effective toric CY orbifold of dimension 3.  

This work takes the next step along this path, developing the analogous orbifold vertex formalism in GW theory.  In doing so, we point out connections to existing theories and applications to conjectural correspondences which will be important motivation for future work.

One possible application of this formalism is to Ruan's crepant resolution conjecture.  The CRC is an extension of the classical Mckay correspondence.  Roughly speaking, the conjecture claims that the GW theory of a Gorenstein orbifold $\X$ and its crepant resolution $Y$ are equivalent after analytic continuation (\cite{cr:qcacr}).  In \cite{bg:crc}, Bryan and Graber conjectured that if $\X$ is hard-Lefshetz, then the equivalence can be witnessed through a change in formal variables and analytic continuation of the GW potentials.

In \cite{cr:ogwcrc}, it was shown that if $\X=[\C^3/\Z_2]$, then the one-leg orbifold vertex associated to $\X$ can be obtained from the open GW invariants of the crepant resolution $\mathcal{K}_{\proj^1}\oplus\mathcal{O}_{\proj^1}$ under a suitable substitution of formal coordinates.  This is the first example of a possible \textit{open} CRC.  It was also shown that this change of variables is compatible with gluing and this was used to deduce the Bryan-Graber formulation of the crepant resolution conjecture for the orbifold $[\so_{\proj^1}(-1)\oplus\so_{\proj^1}(-1)/\Z_2]$.  One hope is that these methods can be generalized to reduce the toric CRC to the case of the orbifold vertex.

Another application is to the orbifold GW/DT correspondence.  One might naturally ask if the DT orbifold vertex of Bryan, Cadman, and Young is equivalent to the GW orbifold vertex defined herein.  The particular case of $[\C^3/\Z_2]$ is the subject of Section \ref{sec:gwdt}.  A future objective is to better understand the connection between these two vertex theories.  A recent paper of Zong \cite{z:gmvf} has generalized the Mari\~no-Vafa formula to the orbifold setting.  One particular consequence is an evaluation of the GW one-leg $A_n$ vertex where the Lagrangian is placed along one of the non-gerby legs.  It would be interesting to compare this evaluation to the analogous DT computation appearing in \cite{bcy:otv}.  It would also be interesting to investigate whether the methods of Zong can be generalized to evaluate the two or three-leg GW orbifold vertex.  We leave these endeavors for future work.

\subsection{Organization of the Paper}

Section \ref{sec:prelim} recalls many of the ideas which will be pivotal throughout.  In particular, we recall basic facts about toric orbifold lines and open GW invariants.  Section \ref{sec:vertex} contains the computations needed for defining the orbifold vertex.  Section \ref{sec:gluing} describes the gluing algorithm and concludes with the proof that the gluing algorithm recovers the GW invariants of any toric CY 3-fold.  Section \ref{sec:insertions} briefly describes how to modify the setup to include insertions.  Section \ref{sec:examples} provides a few of the known examples.  In particular, Section \ref{sec:smooth} recalls the basic facts about the correspondence of the different vertex theories in the smooth case.  Section \ref{sec:apps} gives applications.  In particular, we discuss the conjectural correspondence with the DT orbifold vertex.  Appendix \ref{sec:hhi} is devoted to proving Theorem \ref{hhi}.

\subsection{Acknowledgements} 
A great deal of gratitude is owed to my advisor Renzo Cavalieri.  His guidance, expertise, and enthusiasm made this work possible.  Many thanks also to A. Brini, P. Johnson, C.-C. Liu, and K. Monks for helpful conversation and suggestions.  I would also like to thank the referee for his/her helpful comments and suggestions.\\

\section{Preliminaries}\label{sec:prelim}

\subsection{Toric Calabi-Yau Orbifolds}\label{sec:orbifolds}
By a \textit{Calabi-Yau orbifold} we mean a smooth, quasi-projective Deligne-Mumford stack over $\C$ with trivial canonical class.  We do not require the isotropy of the generic point to be trivial.  A toric Calabi-Yau orbifold is defined to be such a stack with the action of a Deligne-Mumford torus $\mathcal{T}=T\times\mathcal{B}G$ having an open dense orbit isomorphic to $\mathcal{T}$ (c.f. \cite{fmn:stdms}).  To a toric CY orbifold $\X$ of dimension three we can associate a planar graph $\Gamma_{\X}=\{\text{Edges, Vertices}\}$ where the vertices correspond to torus fixed points and the edges correspond to torus invariant lines.  Following \cite{bcy:otv}, we make the following definition.

\begin{definition}\label{orient}
Let $\Gamma$ be a trivalent planar graph with a chosen planar representation.  An \textit{orientation} of $\Gamma$ is a choice of direction for each edge and an ordering of the edges incident to each vertex which is compatible with the counterclockwise cyclic ordering.
\end{definition}

\subsection{The Target Space $[\C^3/G]$}\label{sec:target}

We set up notation here that will be used throughout.  Locally near a torus invariant point of a toric CY 3-fold, the space can be modelled as a global quotient $\left[\C^3/G\right]$ where $G$ preserves the coordinates and acts trivially on the volume form.  We allow $G$ to be any finite abelian group.  More specifically, if $G=\Z_{n_1}\times...\times\Z_{n_l}$, then the action of $G$ on $\C^3$ can be described with \textit{weights} $(\vec{\alpha}^1,\vec{\alpha}^2,\vec{\alpha}^3)\in G^3$ summing to $0$ where the generator $\epsilon_i$ of $\Z_{n_i}$ acts as 
\[
\epsilon_i\cdot (z_1,z_2,z_3)=(e^{\frac{2\pi \sqrt{-1}\alpha_i^1}{n_i}}z_1,e^{\frac{2\pi \sqrt{-1}\alpha_i^2}{n_i}}z_2,e^{\frac{2\pi \sqrt{-1}\alpha_i^3}{n_i}}z_3).
\]
Define
\[
g_i=\text{lcm}\left\{ \frac{n_j}{\gcd\left( \alpha_j^i,n_j \right)}:j=1,...,l \right\}
\]
Then $\Z_{g_i}$ is the effective part of the $G$ action along the $i$th coordinate axis.

We define three Lagrangian suborbifolds $\La_1,\La_2,\La_3$ inside $\left[\C^3/G\right]$ as follows.  We can view $\left[\C^3/G\right]$ as the neigborhood of the (image of) zero in the global quotient 
\begin{equation}\label{doub}
\left[ \so(-1)\oplus\so(-1)/G \right]
\end{equation}
where $z_1$ is the coordinate in the base direction.  Define an anti-holomorphic involution on $\so(-1)\oplus\so(-1)$
\[
\sigma(z_1,z_2,z_3)=(1/\overline{z_1},\overline{z_1}\overline{z_3},\overline{z_1}\overline{z_2}).
\]
One checks that $\sigma$ descends to an involution $\sigma_G$ on the quotient (\ref{doub}).  $\La_1$ is defined to be the fixed locus of $\sigma_G$.  $\La_2$ and $\La_3$ are defined analogously.

The GW orbifold vertex is defined in Definition \ref{def:vertex} to be the oriented open GW potential of a formal neighborhood of the coordinate axes in $\left[\C^3/G\right]$, relative to the Lagrangian $\La:=\La_1\cup \La_2\cup \La_3$.  


\subsection{Toric Orbifold Lines}\label{sec:orbiline}

In order to prove the gluing formula in Section \ref{sec:proof}, we recall some basic facts about toric orbifold lines, i.e. orbifolds with coarse space $\proj^1$.

It follows from the classification theorem of \cite{fmn:stdms} that any toric orbifold line (with finite abelian stabilizers) is an abelian gerbe over a football $\proj_{n_0,n_{\infty}}^1$, and any such orbifold can be constructed via successive root constructions over the football.  Recall that the $n$th root construction of a line bundle $L$ on a space $X$ is defined as the fibered product:

\[\begin{CD}
X^{(L,n)}   	@>>>          \B\C^*       \\ 
@V\psi VV                             @VV\lambda\rightarrow\lambda^nV\\ 
X                  @>L>>      \B\C^*
\end{CD}\]

\noindent where the bottom map classifies the line bundle $L$.  The top map classifies a line bundle $M$ on $X^{(L,n)}$ with $M^{\otimes n}=\psi^*L$.  We denote $M$ by $L^{1/n}$ and refer to it as the $n$th root of $L$.

Generalizing this notion, if $L_1,...,L_l$ are line bundles on $X$ and $n_1,...,n_l\in\mathbb{N}$, then $X^{\mathbf{(L_i,n_i)}}$ can be defined as the fiber product

\[\begin{CD}
X^{\mathbf{(L_i,n_i)}}   	@>>>          \B\C^*\times...\times\B\C^*       \\ 
@V\psi VV                             @VV\lambda_i\rightarrow\lambda_i^{n_i}V\\ 
X                  @>>>      \B\C^*\times...\times\B\C^*
\end{CD}\]

The orbifold Picard group of $X^{\mathbf{(L_i,n_i)}}$ is generated by bundles pulled back from $X$ via $\psi$ along with the $n_i$th root of $L_i$.

We will be concerned with the case when $X$ is the football $\proj_{n_0,n_{\infty}}^1$.  The line bundles on $\proj_{n_{0},n_{\infty}}^1$ are of the form
\begin{equation}\label{line}
L:=\so\left( \frac{a}{n_0}[0]+\frac{b}{n_{\infty}}[\infty]+c \right).
\end{equation}
The \textit{numerical degree} of $L$ is defined to be the Chern-Weil class of the bundle (c.f. \cite{cr:nctoo}).  In this case, we compute $\deg(L)=\frac{a}{n_0}+\frac{b}{n_{\infty}}+c$.  We also make the following definition which will be useful in the computations of Section \ref{sec:gluing}.

\begin{definition}\label{def:int}
Any line bundle on $\proj_{n_0,n_{\infty}}^1$ can be written uniquely in the form (\ref{line}) with $0\leq a\leq n_0-1$ and $0\leq b\leq n_{\infty}-1$.  Given such a representation, we define $\text{Int}(L):=c$.
\end{definition}

Suppose that $L_i=\so\left( \frac{a_i}{n_0}[0]+\frac{b_i}{n_{\infty}}[\infty]+c_i \right)$.  The line bundles on $\proj_{n_0,n_{\infty}}^{\mathbf{(L_i,n_i)}}$ are of the form
\begin{equation}\label{line2}
\psi^*\left(\so\left( \frac{a}{n_0}[0]+\frac{b}{n_{\infty}}[\infty]+c \right)\right)\otimes\left( L_1^{1/n_1} \right)^{m_1}\otimes...\otimes\left( L_l^{1/n_l} \right)^{m_l}.
\end{equation}
We denote such a line bundle by $\so\left( a,b,c;m_1,...,m_l \right)$.  The numerical degree of $\so\left( a,b,c;m_1,...,m_l \right)$ is 
\[
\frac{a}{n_0}+\frac{b}{n_{\infty}}+c+\sum_{i=1}^{l}\frac{m_i}{n_i}\left(\frac{a_i}{n_0}+\frac{b_i}{n_{\infty}}+c_i\right).
\]

An orbifold line bundle contains the information of a representation of the isotropy on the fibers.  For example, in the case of the bundle (\ref{line}) on $\proj_{n_0,n_{\infty}}^1$, the generator of $\Z_{n_0}$ acts on the fiber over $0$ as $\frac{a}{n_0}$ and the generator of $\Z_{n_{\infty}}$ acts on the fiber over $\infty$ as $\frac{b}{n_{\infty}}$ (we have identified $S^1$ with $\R/\Z$).

Following section 2.1.5 of \cite{j:egwtods}, we now give a concrete description of the isotropy of $\proj_{n_0,n_{\infty}}^{\mathbf{(L_i,n_i)}}$ as well as how the groups act on the fibers of the orbifold line bundle $\so\left( a,b,c;m_1,...,m_l \right)$.  Over any point other than $0$ or $\infty$, the isotropy is simply $\Z_{n_1}\times...\times\Z_{n_l}$ and the generator of $\Z_{n_i}$ acts on the fiber of $\so\left( a,b,c;m_1,...,m_l \right)$ by $\frac{m_i}{n_i}$.  Over $0$, the isotropy group $G_0$ is an extension of $\Z_{n_0}$ by $\Z_{n_1}\times...\times\Z_{n_l}$:
\[
0\rightarrow \Z_{n_1}\times...\times\Z_{n_l}\rightarrow G_0\rightarrow \Z_{n_0}\rightarrow 0.
\]
Similarly for the isotropy $G_{\infty}$.  The extension $G_0$ can be determined by the 2-cocycle $\gamma_0:\Z_{n_0}\times\Z_{n_0}\rightarrow \Z_{n_1}\times...\times \Z_{n_l}$ defined by
\[
\gamma_0(r,r')=\begin{cases}
(a_1,...,a_l) &r+r'\geq n_0\\
0 &r+r'<n_0.
\end{cases}
\]
Explicitly, as a set $G_0$ consists of $(l+1)$-tuples $(r,s)$ with $r\in \Z_{n_0}$, $s\in\Z_{n_1}\times...\times\Z_{n_l}$.  The operation is defined by 
\[
(r,s)+(r',s'):=(r+r',s+s'+\gamma_0(r,r')).
\]
On the fiber of $\so(a,b,c;m_1,...,m_l)$ over $0$, we compute that $(1,(0,...,0))\in G_0$ acts as 
\[
\frac{a}{n_0}+ \sum_{i=1}^l\frac{m_ia_i}{n_in_0}
\]
while $(0,(0...,1_i,...,0))\in G_0$ acts as 
\[
\frac{m_i}{n_i}.
\]
Similarly for the action of $G_{\infty}$ on the fiber over $\infty$.

When a rank $2$ bundle $N$ over $\proj_{n_0,n_{\infty}}^{(L_i,n_i)}$ splits as the sum of two line bundles, the Calabi-Yau condition reduces to 
\[
N=\so(a,b,c;m_1,...,m_k)\oplus\so(-a-1,-b-1,-c;-m_1,...,-m_l).
\]

Finally, the following lemma gives us a characterization of when certain torus fixed maps from a football to $\proj_{n_0,n_{\infty}}^{(L_i,n_i)}$ exist.

\begin{lemma}\label{condition}
Suppose $\mathcal{C}$ is an orbifold with coarse space $\proj^1$ and orbifold structure only over $0$ and $\infty$.  Suppose $f:\mathcal{C}\rightarrow \proj_{n_0,n_{\infty}}^{(L_i,n_i)}$ is a $\C^*$ fixed map of degree $d$ twisted at $0$ by $(k_0^0,(k_1^0,...,k_l^0))\in G_0$ and at $\infty$ by $(k_0^{\infty},(k_1^{\infty},...,k_l^{\infty}))\in G_{\infty}$.  Then
\[
k_0^0=k_0^{\infty}=d
\]
and
\[
\hspace{2cm}\frac{dc_i}{n_i}-\frac{k_i^0}{n_i}-\frac{k_i^{\infty}}{n_i}\in\Z \hspace{1cm}\forall i\geq 1.
\]
\end{lemma}

\begin{proof}
This is Lemmas II.12 and II.13 in \cite{j:egwtods}.
\end{proof}

In particular, once the degree and twisting at either $0$ or $\infty$ are fixed, then the other twisting is determined. 

\subsection{Open Gromov-Witten Invariants}  In \cite{kl:oinv}, Katz and Liu propose a tangent obstruction theory for the moduli space $\M_{g;h}(X,L|d;\gamma_1,...,\gamma_h)$ of stable maps from $h$-holed Riemann surfaces into $X$ with degree given by $d\in H_2(X,L;\Z)$ and boundary conditions given by $\gamma_i\in H_1(L;\Z)$ provided the following two conditions are met:
\begin{itemize}
\item $L$ is the fixed locus of an anti-holomorphic involution $\sigma:X\rightarrow X$ and
\item $(X,L)$ can be equipped with a well-behaved $\C^*$ action with the real subtorus $S^1$ fixing $L$.
\end{itemize}
An important aspect of their theory is that varying the choice of torus action leads to a family of localized virtual fundamental classes.

In particular, if $X=\text{Tot}(\so_{\proj^1}(-1)\oplus\so_{\proj^1}(-1))$, $L$ is the fixed locus of the involution $\sigma:(z,u,v)\rightarrow (1/\bar{z},\bar{z}\bar{v},\bar{z}\bar{u})$, and we consider degree $d$ maps from a disk $(D^2,S^1)$, then the proposed virtual fundamental class is
\[
e(R^1\pi_*f^*N(d))\cap\M_{0,1}(D^2,S^1|d)
\]
where $N(d)$ is a Riemann-Hilbert bundle defined in Example 3.4.4 of \cite{kl:oinv}.  Assuming that the virtual localization formula of \cite{Graber-Pandharipande} naturally generalizes to this setting, Katz and Liu suggested that the contribution of such a disk invariant to the open GW potential of $(X,L)$ should be given by
\[
\frac{t}{d^2}e^{eq}(-R^{\bullet}\pi_*f^*(T_X,T_L))=\frac{t}{d^2}\frac{e^{eq}H^1(N(d))}{e^{eq}H^0(L(2d))}
\]
where $L(2d)$ is defined in Example 3.4.3 of \cite{kl:oinv}, the Euler classes are $S^1$ equivariant, and $t$ is generator of $H_{S^1}^*(\{pt\})$.  The motivation for $N(d)$ and $L(2d)$ is that their holomorphic doubles are precisely the bundles $\so_{\proj^1}(-d)\oplus\so_{\proj^1}(-d)$ and $\so_{\proj^1}(2d)$, respectively.  In other words, the Riemann-Hilbert bundles are `half' of the bundles obtained by pulling back $T_X$ along a degree $d$ map to the base $\proj^1$.

We generalize this setup in two ways.  First, we observe that the disk invariants of Katz and Liu are local to the neighborhood of $0\in\proj^1$ and should therefore be regarded as disk invariants of $(\C^3,L)$.  Second, we generalize the setup to the case where the target (and thus the source) are orbifolds.  In order to make this generalization, we define certain orbifold Riemann-Hilbert bundles in the next section which naturally generalize $L(2d)$ and $N(d)$.

\subsection{Orbifold Riemann-Hilbert Bundles}\label{sec:rhbundles}

We define two classes of Riemann-Hilbert bundles on the orbifold disks $(D_r,S^1):=[(D^2,S^1)/\Z_r]$ which play a crucial role in the computation of the open GW potential of $\left(\left[\C^3/G\right],\La\right)$.

\subsubsection{$L(m)$:}\label{rh1}
For $m>0$, consider the bundle $\so(m,m,0)$ on the football $\proj_{r,r}^1$.  There is an anti-holomorphic involution 
\[
\sigma:(z,u)\rightarrow(1/\overline{z}, -\overline{z}^{-2m/r}\overline{u})
\]
where $z,u$ are local coordinates near $0$ on the coarse space.  We define a Riemann-Hilbert bundle $L(m)$ on $(D_r,S^1)$ by
\[
L(m):=(\so(m,m,0)_{|_{D_r}},\so(m,m,0)_{|_{S^1}}^{\sigma}).
\]

Then the global sections of $L(m)$ are by definition the $\sigma$ invariant global sections of $\so(m,m,0)$.  The global sections of $\so(m,m,0)$ are generated by
\[
\left\{ z^{\langle \frac{m}{r} \rangle+i} \right\}_{i=0}^{2\lfloor\frac{m}{r}\rfloor}.
\]
Since 
\[
\sigma:\left(z,\alpha_iz^{\langle\frac{m}{r}\rangle+i}\right)\rightarrow \left(z,-\overline{\alpha}_iz^{2\lfloor\frac{m}{r}\rfloor+\langle \frac{m}{r} \rangle-i}\right),
\]
then the global sections of $L(m)$ are:
\begin{equation}\label{h0sections}
\sum_{i=0}^{\lfloor \frac{m}{r} \rfloor-1}\left(\alpha_iz^{\langle\frac{m}{r}\rangle+i}-\overline{\alpha}_iz^{2\lfloor\frac{m}{r}\rfloor+\langle \frac{m}{r} \rangle-i}\right)+\sqrt{-1}\beta z^{\frac{m}{r}}
\end{equation}
with $\alpha_i\in\C$ and $\beta\in\R$.  We can embed the sections (\ref{h0sections}) torus equivariantly into the global sections of $\so(m,0,0)$ by mapping them to
\[
\sum_{i=0}^{\lfloor \frac{m}{r} \rfloor-1}\alpha_iz^{\langle\frac{m}{r}\rangle+i}+\beta z^{\frac{m}{r}}.
\]
\begin{remark}
This natural choice of identification determines an orientation for the sections.
\end{remark}

\subsubsection{$N(m,n,l)$:}\label{rh2}
Given $m,n\in\Z_{\geq 0}$ and $l\in\Q_{\geq 0}$ with $-l+\frac{n}{r}\in\Z$, consider the rank $2$ bundle \[N_1\oplus N_2=\so(m,-m-n,-l+\frac{n}{r})\oplus\so(-m-n,m,-l+\frac{n}{r})\] on $\proj_{r,r}^1$.  There is an anti-holomorphic involution \[\sigma:(z,u,v)\rightarrow(1/\overline{z},\overline{z}^{l}\overline{v},\overline{z}^{l}\overline{u}).\]
We define the Riemann-Hilbert bundle $N(m,n,l)$ on $(D_r,S^1)$ by:
\[
N(m,n,l):=\left(\left(N_1\oplus N_2\right)_{|_{D_r}},\left(N_1\oplus N_2\right)_{|_{S^1}}^{\sigma}\right).
\]

The $H^1$ sections of $N(m,n,l)$ are by definition the $\sigma$ invariant $H^1$ sections of $N_1\oplus N_2$.  
The $H^1$ sections of $N_1\oplus N_2$ are generated by
\[
\left\{ \left( z^{\langle \frac{m}{r} \rangle-i}, z^{\langle \frac{-m-n}{r}\rangle-j} \right) \right\}_{i,j=1}^{l+\langle\frac{m}{r}\rangle+\langle\frac{-m-n}{r}\rangle-1}.
\]
Since
\[
\sigma:\left( z,\alpha_iz^{\langle \frac{m}{r} \rangle-i}, \beta_jz^{\langle \frac{-m-n}{r}\rangle-j}\right)\rightarrow\left(z, \overline{\beta}_jz^{-l-\langle\frac{-m-n}{r}\rangle+j},\overline{\alpha}_iz^{-l-\langle\frac{m}{r}\rangle+i}  \right),
\]
the $\sigma$ invariant sections are:
\[
\left( \sum_{i=1}^{l+\langle \frac{m}{r} \rangle+\langle \frac{-m-n}{r} \rangle-1}\frac{\alpha_i}{z^{i-\langle \frac{m}{r} \rangle}},\sum_{i=1}^{l+\langle \frac{m}{r} \rangle+\langle \frac{-m-n}{r} \rangle-1}\frac{\overline{\alpha}_i}{z^{l+\langle \frac{m}{r} \rangle-i}}\right)
\]
\begin{remark}\label{orientation}
In order to compute the equivariant Euler class of these bundles, we face the choice of embedding the sections into either $N_1$ or $N_2$.  We will denote the corresponding Euler classes by $e(H^1(N(m,n,l)))^+$ or $e(H^1(N(m,n,l)))^-$, respectively.
\end{remark}

\section{The Orbifold Vertex}\label{sec:vertex}

In this section, the orbifold vertex is defined via localization.  The result is an expression in terms of $G$-Hodge integrals on $\M_{g,n}(\B G)$ and a combinatorial disk function.

\begin{figure}
\vspace{-2cm}\includegraphics[height=4.5cm]{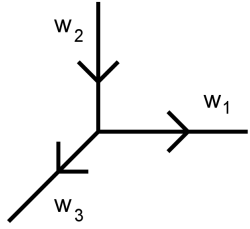}\vspace{-1cm}
\caption{The toric diagram $\Gamma_{\X}$.  The labelling of the weights coincides with the counterclockwise orientation of the coordinates.  The orientations of the edges show that disk invariants are computed with $D^+$ on the $1$st and $3$rd coordinates and $D^-$ on the $2$nd.}
\label{localvert}
\end{figure}

\subsection{Torus Action and Fixed Loci}

Set $\X:=[\C^3/G]$ and choose an orientation for $\Gamma_{\X}$ consistent with the labelling $(x_1,x_2,x_3)$ of the coarse coordinates for $\X$ (see Figure \ref{localvert}).  To equip $\X$ with a $\C^*$ action, begin with an action on $\C^3$ having CY weights \[\vec{w}:=\left(w_1,w_2,w_3\right) \text{ with } w_1+w_2+w_3=0.\]  This action descends to the quotient and the corresponding weights on the coordinates of the coarse space are $g_1w_1,g_2w_2,g_3w_3$.  One checks that $S^1\subset \C^*$ fixes each $\La_i$.

Pulling the torus action back to the moduli space of stable bordered maps, the $S^1$ fixed loci consist of maps $f:\Sigma\rightarrow \X$ such that
\begin{itemize}
\item $\Sigma$ is a compact curve attached to some number of disks at possibly twisted nodes,
\item $f$ contracts the compact curve to the origin,
\item $f$ maps a disk onto the $i$th axis with local expression $z\rightarrow z^d$ ($d$ is the winding number) and with boundary mapping to \[\La_i\cap\{x_j=0|j\neq i\}\cong S^1.\]  
\end{itemize}

Such a fixed locus can be encoded with the datum of:
\begin{itemize}
\item The genus of the contracted curve,
\item Winding profiles: Three vectors of positive integers, $\vec{d^i}=(d_1^i,...,d_{l_i}^i)$ $i=1,2,3$, corresponding to $l_i$ disks mapping to $\La_i$ with windings determined by $d_j^i$, and
\item Twisting profiles: Three vectors of elements of $G$, $\vec{k^i}=(k_1^i,...,k_{l_i}^i)$ $i=1,2,3$, corresponding to the twisting of the nodes attaching the corresponding disk to the contracted curve.
\end{itemize}

In order to more easily track such a locus, set $\vec{\mu}=(\mu^1,\mu^2,\mu^3)$ where
\begin{equation}\label{mu}
\mu^i=\left\{ (d_1^i,k_1^i),...,(d_{l_i}^i,k_{l_i}^i) \right\}.
\end{equation}

\begin{remark}
$\mu^i$ can naturally be identified with a conjugacy classes in the wreath product $G\wr S_d$ with $d=\sum_j d_j^i$ (c.f. \cite{m:sfhp}).
\end{remark}

We denote such a fixed locus by $F(g,\vec\mu)$.  $F(g,\vec\mu)$ can be identified with a product of a zero dimensional stack and $\M_{g,n}\left(\B G\right)\cap ev^*(-\vec{k})$ where $n=l_1+l_2+l_3$ and $ev^*(-\vec{k})$ is shorthand for pulling back all of the $-k_j^i$-twisted points via the appropriate evaluation maps.  The zero-dimensional stack encodes automorphisms of the fixed loci which will be accounted for in Section \ref{sec:ob}

The next lemma gives a condition on each disk which is necessary for the above locus to be nonempty (c.f. \cite{bc:ooinv}, section 2.2.2).

\begin{lemma}\label{condition2}
Suppose $h:(D_r,S^1)\rightarrow (\X,\La_i)$ is a $\C^*$ fixed winding $d$ map, twisted by $k=(k_1,...,k_l)\in G$.  Then
\[
\frac{-d}{g_i}+\sum_j\frac{k_j\alpha_j^i}{n_j}\in\Z.
\]
\end{lemma}

\begin{proof}
Doubling the map $h$, we get a degree $d$ map $\hat{h}:\proj_{r,r}^1\rightarrow [\so(-1)\oplus\so(-1)/G]$ which is twisted by $k$ at $0$ and $-k$ at $\infty$.  The target is a rank $2$ bundle $N_1\oplus N_2$ over $[\proj^1/G]$.  Consider the bundle $\hat{h}^*N_1$ on $\proj_{r,r}^1$.  Define
\[
r_{i}:=r\sum_j\frac{k_j\alpha_j^{i}}{n_j}
\]
Then the generator of the isotropy on $\proj_{r,r}^1$ at $0$ acts on the fiber of the pullback by $\frac{r_{i+1}}{r}$ and the generator of the isotropy at $\infty$ acts by $\frac{-r_i-r_{i+1}}{r}$.  Furthermore, the degree of $\hat{h}^*N_1$ is $\frac{-d}{g_i}$.  Since the isotropy at $0$ contributes a fractional part of $\langle \frac{r_{i+1}}{r} \rangle$ to the degree and the isotropy at $\infty$ contributes a fractional part of $\langle \frac{-r_i-r_{i+1}}{r} \rangle$ to the degree, we see that
\begin{equation}\label{lem}
\frac{-d}{g_i}-\left\langle \frac{r_{i+1}}{r} \right\rangle-\left\langle \frac{-r_i-r_{i+1}}{r} \right\rangle\in\Z.
\end{equation}
The result follows from (\ref{lem}).
\end{proof}

\subsection{The Obstruction Theory}\label{sec:ob}

In this section, we give a precise formula for the restriction of the obstruction theory to a fixed locus in terms of the combinatorial data of that locus.  We parse the contributions into three components.

\begin{itemize}
\item \textbf{Compact Curve}:  A contracting compact curve contributes a Hodge part:
\[
e^{eq}\left( \mathbb{E}_{\vec{\alpha}^1}^{\vee}(w_1)\oplus\mathbb{E}_{\vec{\alpha}^2}^{\vee}(w_2)\oplus\mathbb{E}_{\vec{\alpha}^3}^{\vee}(w_3) \right)
\]
where $\mathbb{E}_{\vec{\alpha}^i}^{\vee}(w_i)$ denotes the dual of the $\alpha^i$-character sub-bundle twisted by the torus weight $w_i$.  The contracting curve also contributes a factor of $w_i$ for each direction in which the curve can be perturbed, i.e. if the contraction to $\B G$ factors through $\B G_i$ where $G_i$ is the isotropy along the $i$th axis.  We denote this contribution by $\gamma$.

\item \textbf{Nodes}: Each node contributes a gluing factor of $|G|$ explained in Section 2.1 of \cite{cc:gl}.  There is a contribution of $w_i$ for each direction which the node can be perturbed.  There appears a node smoothing contribution of
\[
\left(\frac{ g_iw_i}{d_j^i}-\psi_j^i\right)^{-1}.
\]
There is a term 
\[
\left(\frac{w_ig_i}{d_j^i}\right)^{-\delta_{0,k_j^i}}
\]
to cancel the infinitesimal automorphism at the origin of the disk.  There is also a contribution of $w_l$ for each direction which the node can be perturbed.  This contributions is
\[
\delta(k_j^i):= \prod\{w_l:k_j^i\in G_l\}.
\]

\item \textbf{Disks}: A disk mapping to the $i$th leg with winding $d$ and twisting $k\in G$ contributes a combinatorial factor of $D_k^{\sigma_i}(d;\vec{w})$ which is described in the next section.
\end{itemize}

\begin{remark}
We have abused notation at this point.  Having previously defined the $w_i$ to be numbers we are now treating them as equivariant cohomology classes.  We have done this to ease notation and one can simply interpret the $w_i$ in this context as $w_it$ where $t$ is the generator of $H_{S^1}^*(\{pt\})$.
\end{remark}

\subsection{The Disk Function}

Suppose $h:(D_r,S^1)\rightarrow(\X,\La_i)$ is a $k=(k_1,...,k_l)$ twisted, winding $d$ map from an orbidisk.  As in Lemma (\ref{condition}), define integers
\[
r_i:=r\sum_j\frac{k_j\alpha_j^i}{n_j} \text{ for } i=1,2,3.
\]
Then the generator of the isotropy of $D_r$ acts on the fiber of the pullback of the three normal directions to the origin in $\X$ as $\frac{r_i}{r}$.

The disk has $d\cdot\frac{|G|}{g_i}$ global automorphisms and an infinitesimal automorphism factor $\frac{w_ig_i}{d}$.  We define
\[
D_k^{\sigma_i}(d;\vec{w}):=\left(\frac{g_iw_i}{d}\right)^{\delta_{0,k}}\frac{ g_i}{d|G|} \frac{e^{eq}\left(H^1N\left(r_{i+1}, r_i, \frac{d}{g_i}\right)\right)^{\sigma_i}}{e^{eq}\left(H^0L\left( \frac{rd}{g_i} \right)\right)}
\]
where
\[
\sigma_i:=\begin{cases}
+ & \text{if the $i$th leg is oriented outward,}\\
- & \text{if the $i$th leg is oriented inward.}
\end{cases}
\]

\begin{remark}
As a notational convenience, we define $r_4:=r_1$ and $r_5:=r_2$ to reflect the cyclic labeling of the coordinates.
\end{remark}

\begin{remark}
We can define $N\left(r_{i+1}, r_i, \frac{d}{g_i}\right)$ (i.e. it satisfies the hypothesis of Example \ref{rh2}) due to Lemma \ref{condition2}.
\end{remark}

Making the identifications of Example \ref{rh1}, we see that
\begin{align*}
e^{eq}\left(H^0L\left(\frac{rd}{g_i}\right)\right)&=\prod_{j=0}^{\left\lfloor \frac{d}{g_i} \right\rfloor-1}\left(  w_i-\frac{ g_iw_i}{d}\left( \left\langle \frac{d}{g_i} \right\rangle+j \right) \right)\\
&=\prod_{j=1}^{\left\lfloor \frac{d}{g_i}\right\rfloor}\frac{ g_iw_i}{d}j\\
&= \left( \frac{ g_iw_i}{d} \right)^{\left\lfloor \frac{d}{g_i} \right\rfloor} \left\lfloor \frac{d}{g_i} \right\rfloor !
\end{align*}
where we have left out the weight $0$ contribution from $h^*(\frac{\partial}{\partial z_i})$ as usual (c.f. section 27.4 of \cite{clay:ms}).  Choosing the `positive' orientation of sections discussed in Example \ref{rh2}, we compute
\begin{align*}
e^{eq}\left(H^1N\left(r_{i+1}, r_i, \frac{d}{g_i}\right)\right)^+&=\prod_{j=1}^{\frac{d}{g_i}+\left\langle \frac{r_{i+1}}{r} \right\rangle+\left\langle \frac{r_{i+2}}{r} \right\rangle-1}\left(  w_{i+1}+\frac{ g_iw_i}{d}\left( j- \left\langle \frac{r_{i+1}}{r} \right\rangle \right) \right)\\
&=\left( \frac{ g_iw_i}{d} \right)^{\frac{d}{g_i}+\left\langle \frac{r_{i+1}}{r} \right\rangle+\left\langle \frac{r_{i+2}}{r} \right\rangle-1}\frac{\Gamma\left( \frac{dw_{i+1}}{g_iw_i}+\left\langle \frac{r_{i+2}}{r} \right\rangle + \frac{d}{g_i} \right)}{\Gamma\left( \frac{dw_{i+1}}{g_iw_i}-\left\langle \frac{r_{i+1}}{r} \right\rangle +1 \right)}.
\end{align*}
One can check (using $w_1+w_2+w_3=0$) that the `negative' orientation of the sections merely introduces a factor of
\[
(-1)^{\frac{d}{g_i}+\left\langle \frac{r_{i+1}}{r} \right\rangle+\left\langle \frac{r_{i+2}}{r} \right\rangle-1}.
\]
Putting this all together, we compute
\[
D_k^+(d;\vec{w})=\left( \frac{ g_iw_i}{d} \right)^{\text{age}(k)+\delta_{0,k}-1}\frac{g_i}{d|G|\left\lfloor \frac{d}{g_i} \right\rfloor !}\frac{\Gamma\left( \frac{dw_{i+1}}{g_iw_i}+\left\langle \frac{r_{i+2}}{r} \right\rangle + \frac{d}{g_i} \right)}{\Gamma\left( \frac{dw_{i+1}}{g_iw_i}-\left\langle \frac{r_{i+1}}{r} \right\rangle +1 \right)}.
\]

\begin{remark}
This generalizes the disk function defined in Section 2.2.3 of \cite{bc:ooinv}.
\end{remark}

\subsection{The Orbifold Vertex}

We now put together the contributions coming from the compact curve, the nodes, and the disks.  We assign formal variables to track the discrete invariants.  We let $\lambda$ track the Euler characteristic $2g-2+n$.  For $\vec\mu$ as in (\ref{mu}), define the formal variables
\[
\bold{p}_{\vec\mu}:=p_{\mu^1}^1\cdot p_{\mu^2}^2\cdot p_{\mu^3}^3
\]
with formal multiplication given by
\[
\left(p^1_{\mu^1}\cdot p^2_{\mu^2}\cdot p^3_{\mu^2}\right)\cdot\left(p^1_{\nu^1}\cdot p^2_{\nu^2}\cdot p^3_{\nu^2}\right) := \left(p^1_{\mu^1\cup\nu^1}\cdot p^2_{\mu^2\cup\nu^2}\cdot p^3_{\mu^2\cup\nu^3}\right).
\] 

\begin{definition}\label{def:vertex}
Define
\begin{align*}
\hspace{-1cm}V_{\X,g,\vec\mu}(\vec{w}):=&\frac{\prod \left(|G|\delta(k_j^i)\left(\frac{ g_iw_i}{d}\right)^{-\delta_{0,k_j^i}}D_{k_j^i}^{\sigma_i}(d_j^i;\vec{w})\right)}{\gamma\cdot|\text{Aut}(\vec{\mu})|} 
\label{def:vertex1}\\
&\cdot\int_{\M_{g,n}\left(\B G\right)\cap ev^*\left(-\vec{k}\right)}\frac{e^{eq}\left( \mathbb{E}_{\vec{\alpha}^1}^{\vee}(w_1)\oplus\mathbb{E}_{\vec{\alpha}^2}^{\vee}(w_2)\oplus\mathbb{E}_{\vec{\alpha}^3}^{\vee}(w_3) \right)}{\prod \left(\frac{ g_iw_i}{d_j^i}-\psi_j^i\right)}
\end{align*}
where
\[
D_k^{+}(d;\vec{w})=\left( \frac{ g_iw_i}{d} \right)^{\text{age}(k)+\delta_{0,k}-1}\frac{g_i}{d|G|\left\lfloor \frac{d}{g_i} \right\rfloor !}\frac{\Gamma\left( \frac{dw_{i+1}}{g_iw_i}+\left\langle \frac{r_{i+2}}{r} \right\rangle + \frac{d}{g_i} \right)}{\Gamma\left( \frac{dw_{i+1}}{g_iw_i}-\left\langle \frac{r_{i+1}}{r} \right\rangle +1 \right)}
\]
and
\[
D_k^{-}(d;\vec{w})=(-1)^{\frac{d}{g_i}+\left\langle \frac{r_{i+1}}{r} \right\rangle+\left\langle \frac{r_{i+2}}{r} \right\rangle-1}D_k^{+}(d;\vec{w}).
\]

We define the \textit{connected open GW potential} of $\X$ by
\begin{equation*}\label{def:vertex2}
V_\X(\lambda,\bold{p};\vec{w}):=\sum_{g,\vec\mu} V_{\X,g,\vec\mu}(\vec{w})\lambda^{2g-2+n}\bold{p}_{\vec\mu}.
\end{equation*}
We define the \textit{disconnected open GW potential} of $\X$ by
\begin{equation*}\label{def:vertex3}
V_\X^{\bullet}(\lambda,\bold{p};\vec{w}):=\exp\left( V_\X(\lambda,\bold{p};\vec{w}) \right).
\end{equation*}
Finally, we define the \textit{oriented GW orbifold topological vertex} $V_{\X,\vec\mu}^{\bullet}(\lambda;\vec{w})$ to be the coefficient of $\mathbf{p}_{\vec\mu}$ in $V_\X^{\bullet}(\lambda,\bold{p};\vec{w})$.

\end{definition}

\begin{remark} If the moduli space in the above definition does not exist, we set the contribution equal to zero except in two exceptional cases where we make the following conventions for the unstable integrals.  We set
\[
\int_{\M_{0,1}(\B G)\cap ev^*(0)}\frac{1}{a-\psi}:=\frac{a}{|G|}
\]
and
\[
\int_{\M_{0,2}(\B G)\cap ev_i^*(k_i)}\frac{1}{(a_1-\psi_1)(a_2-\psi_2)}:=\begin{cases}
 \frac{1}{|G|\left(a_1+a_2\right)} &\text{ if } k_1=-k_2 \in G\\
 0 &\text{ else.}
\end{cases}
\]
\end{remark}

\section{Gluing Algorithm}\label{sec:gluing}

In this section, we describe an algorithm for gluing open amplitudes.  We prove that the GW orbifold vertex and the gluing algorithm determine the GW invariants of any toric CY orbifold of dimension 3.


\subsection{The Geometric Setup}

Let $\Y$ be a toric Calabi-Yau orbifold of dimension $3$ and $\Gamma_{\Y}$ the corresponding toric diagram with a chosen orientation as in Definition \ref{orient}.  Let $V:=\{v_i\}$ be the set of vertices in $\Gamma$ and let $G_i$ be the isotropy group at the corresponding point $y_i\in\Y$.  We set $E^c:=\{e_{ij}\}$ to be the compact edges in $\Gamma_{\Y}$ directed from $v_i$ to $v_j$.  Let $\mathcal{C}_{ij}$ denote the corresponding line connecting $y_i$ to $y_j$.  If $e\in E^{c}$, let $C_e$ be the corresponding curve in $\Y$, let $D_e\in H^2(\Y,\Q)$ be the dual of $[C_e]\in H_2(\Y,\Q)$ under the natural pairing, and let $G_e$ be the isotropy over $C_e$.  Define $N_{e,r}$ ($N_{e,l}$) to be the orbifold line bundle on $C_e$ corresponding to the toric divisor to the right (left) of $e$.


Choose a Calabi-Yau $\C^*$ action on $\Y$.  Let $\Y_i$ be a neighborhood of $y_i$ so that $\Y_i\cong [\C^3/G_i]$.  Let $(x_1^i,x_2^i,x_3^i)$ be the coarse coordinates at $\Y_i$ so that the cyclic ordering of the coordinates coincides with the orientation of $\Gamma_{\Y}$ at $y_i$.  These coordinates inherit an orientation and a $\C^*$ action from $\Y$, we label the weights of the action $\vec{w}(i):=(w_1^i,w_2^i,w_3^i)$.  We define Lagrangians in $\Y_i$ as in Section \ref{sec:vertex}.  This gives us everything we need to compute amplitudes as in Definition \ref{def:vertex}.  


\subsection{Edge Assignments}\label{sec:edge}



\begin{definition}\label{def:pairing}
We say that the triple $(d,k,k')\in \N\times G_i\times G_{i'}$ is \textit{admissible with respect to $e_{ii'}$} if the map 
\[
f:\proj_{r,s}^1\rightarrow\mathcal{C}_{ii'}
\]
given by $z\rightarrow z^d$ and twisted by $k$ at $y_i$ and $k'$ at $y_{i'}$ exists.  In this case, define
\[
n(e,d,k,k'):=\text{Int}(f^*N_{e,r})+1
\]
(c.f. Definition \ref{def:int}).
\end{definition}



\begin{remark}
Lemma \ref{condition} gives a characterization of admissible triples.
\end{remark}


\begin{definition}
An \textit{edge assignment} for $\Gamma_{\Y}$ is a finite set of admissible triples for each $e\in E^c$.  
Let $A_{\Y}$ be the set of edge assignments.  For $\Lambda\in A_\Y$, $e\in E^c$, define $\Lambda_e$ to be the set of admissible triples over $e$.  Define 
\[
n(\Lambda_e):=\sum_{\Lambda_e}n(e,d,k,k').
\]
Define $\rho_d^k(\Lambda_e)$ to be the number of times the triple $(d,k,k')$ appears in $\Lambda_e$.  Also define
\[
d(\Lambda_e):=\sum_{\Lambda_e}d.
\]
Finally, define $\Lambda_i$ to be the induced triple of twisting/winding profiles at $y_i$.
\end{definition}

\subsection{Gluing Algorithm}

We are now ready to state the main gluing algorithm.
\begin{theorem}\label{gluing}
The Gromov-Witten potential of $\Y$ is
\[
Z_{\Y}^\bullet:=\sum_{\Lambda\in A_\Y}\prod_{v_i\in V}V_{\Y_i,\Lambda_i}^\bullet(\lambda;\vec{w}(i))\prod_{e\in E^c}\prod_{d,k}\left( d|G_e| \right)^{\rho_k^d(\Lambda_e)}\rho_k^d(\Lambda_e)!(-1)^{n(\Lambda_e)}Q_e^{d(\Lambda_e)}
\]
where $Q_e$ are formal variables tracking the degree and we impose $Q_e=Q_e'$ if $[C_e]=[C_e']\in H_2(\Y)$.
\end{theorem}

\begin{remark}
It is a consequence of the theorem that $Z_{\Y}^\bullet$ is independent of the choices of orientation and torus action.
\end{remark}

\begin{remark}
The obvious extension of this algorithm can be used to define/compute the open GW potential of $\Y$ with inner and/or outer branes.  These potentials will depend on the choices of orientation and torus action.
\end{remark}

\subsection{Proof of Gluing Formula}\label{sec:proof}

By localization arguments, it need only be checked that disk contributions glue to multiple cover contributions on a given edge.  Using the notation of Section \ref{sec:orbiline}, a given edge is isomorphic to $\proj_{r,s}^{\mathbf{(L_i,n_i)}}$ with isotropy group $G_0$ at $0$ and $G_{\infty}$ at $\infty$.  To prove the gluing formula, we compute the contribution to the potential from a representable $\C^*$ fixed map from a football $f:\F\rightarrow \proj_{r,s}^{\mathbf{(L_i,n_i)}}\subset\Y$ twisted by $\vec{k}^0=(d,(k_1^0,...,k_l^0))\in G_0$ at $0$ and $\vec{k}^{\infty}=(d,(k_1^{\infty},...,k_l^{\infty}))\in G_{\infty}$ at $\infty$.  We show that this contribution decomposes as the corresponding disk contributions along with the prescribed gluing factor.

Since $\Y$ is a toric Calabi-Yau 3-fold, the normal bundle splits as
\begin{align*}
N_{\proj_{r,s}^{\mathbf{(L_i,n_i)}}/\Y}&=N_l\oplus N_r\\
&\hspace{-1cm}=\so(a,b,c;m_1,...,m_l)\oplus\so(-a-1,-b-1,-c;-m_1,...,-m_l).
\end{align*}
(c.f. Section \ref{sec:orbiline})

The total space of the normal bundle inherits an orientation and a Calabi-Yau $\C^*$ action from $\Y$.  We label the oriented weights of this action (on $\C^3$) by $w_1, w_2, w_3$ at $0$ and
\[
w_1':=\frac{-rw_1}{s},
\]
\[
w_2':=w_3-rw_1\left( \frac{-a-1}{r}+\frac{-b-1}{s}-c - \sum m_i\left( \frac{a_i}{rn_i}+\frac{b_i}{sn_i}+\frac{c_i}{n_i} \right) \right),
\]
and
\[
w_3':=w_2-rw_1\left( \frac{a}{r}+\frac{b}{s}+c+\sum m_i\left( \frac{a_i}{rn_i}+\frac{b_i}{sn_i}+\frac{c_i}{n_i} \right) \right)
\]
at $\infty$.  Near $\proj_{r,s}^{\mathbf{(L_i,n_i)}}\subset\Y$, $\Gamma_{\Y}$ can be decorated as in Figure \ref{localline} to account for the orientation and the $\C^*$ action.

\begin{figure}
\vspace{-2cm}\includegraphics[height=5cm]{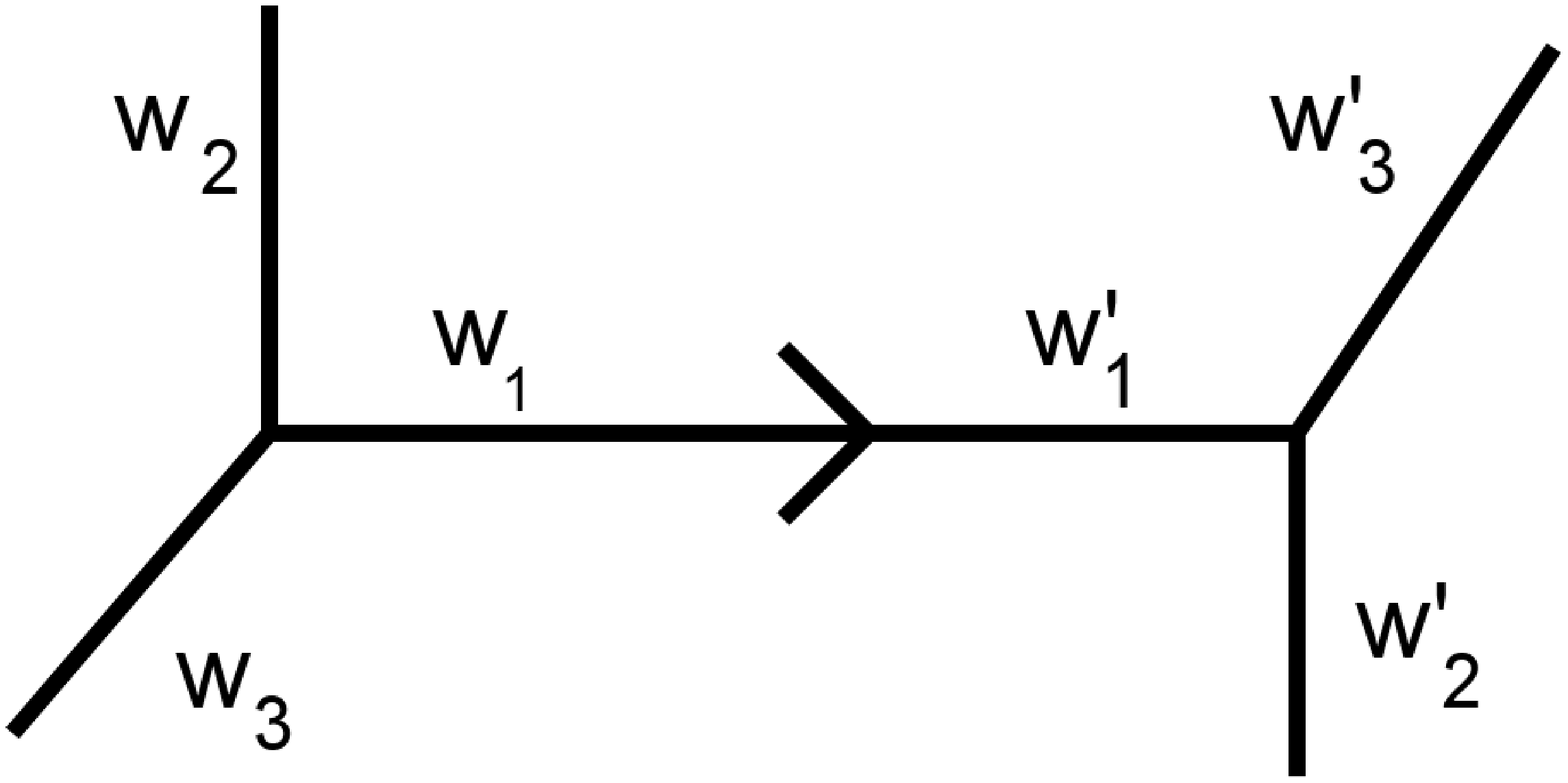}\vspace{-1.5cm}
\caption{The neighborhood of $\proj_{r,s}^{\mathbf{(L_i,n_i)}}\subset\Y$}
\label{localline}
\end{figure}

Using the usual obstruction theory for local invariants and the virtual localization formula, the contribution of the map $f$ is
\begin{equation}\label{eqn:footballcontrib}
\frac{w_1w_1'}{\tau_f\tau'_fd^3n}\frac{e^{eq}\left( -R^{\bullet}\pi_*f^*N_l \right)\cdot e^{eq}\left(-R^{\bullet}\pi_*f^*N_r \right)}{ e^{eq}\left(R^0\pi_*f^*T_{\proj_{r,s}^{\mathbf{(L_i,n_i)}}}\right)}
\end{equation}
where $n=\sum n_i$ and as before $\tau_f$ and $\tau'_f$ cancel the infinitesimal automorphism if $0$ and/or $\infty$ is either marked or a node.

\begin{remark}\label{convention}
In the computations that follow, we make the following index and product convention:  Suppose $m < n$, set
\[
\{A_i\}_{i=n}^{m}=\{A_{m+1},A_{m+2},...,A_{n-1}\},
\]
i.e. we forget the first and last terms if the index decreases.  Similarly, we set
\[
\prod_{i=n}^{m}A_i=\left(A_{m+1}\cdot A_{m+2}\cdot...\cdot A_{n-1}\right)^{-1}.
\]
If the index increases, we interpret the sets and products as usual.  The only scope of this convention is to make the following exposition more compact.
\end{remark}

\subsubsection{Numerator of (\ref{eqn:footballcontrib})}

We begin by computing $e^{eq}\left(R^{\bullet}\pi_*f^*N_l \right)$.  Suppose $s_0$ is a minimally vanishing section of $f^*N_1$ on the complement of $\infty$ and $s_{\infty}$ is a minimally vanishing section on the complement of $0$.  Then if $z$ is a coarse local coordinate of $\mathcal{F}$ at $0$,
\[
s_{\infty}=s_0z^{M}
\]
where
\[
M:=\text{Int}(f^*N_l)=\deg(f^*N_l)-\text{ord}_0(s_0)-\text{ord}_{\infty}(s_{\infty}).
\]
Therefore the vector space $H^*(\F,f^*N_l)$ is generated by $\{s_0z^p\}_{p=0}^M$ with equivariant weights
\[
\left\{ W_1(p):=\left( w_2-\frac{ rw_1}{d}\left( \text{ord}_0(s_0)+p \right)\right) \right\}_{p=0}^M.
\]
Similarly, if we let $t_0$ and $t_{\infty}$ be minimally vanishing sections of $f^*N_r$, we see that
\[
t_{\infty}=t_0z^M
\]
where 
\[
N:=\text{Int}(f^*N_r)=\deg(f^*N_r)-\text{ord}_0(t_0)-\text{ord}_{\infty}(t_{\infty}).
\]
Therefore the vector space $H^*(\F,f^*N_r)$ is generated by $\{t_0z^q\}_{q=0}^N$ with equivariant weights
\[
\left\{W_2(q):= \left( w_3-\frac{ rw_1}{d}\left( \text{ord}_0(t_0)+q \right)\right) \right\}_{q=0}^N.
\]
Multiplying the classes, we get
\[
e^{eq}\left( R^{\bullet}\pi_*f^*N_{\proj_{r,s}^{\mathbf{(L_i,n_i)}}/\Y}\right)=\prod_{p=0}^M W_1(p) \prod_{q=0}^N W_2(q).
\]

Define positive integers
\[
\hat{r}:=\frac{d}{r}+\ord_0(s_0)+\ord_0(t_0)
\]
and
\[
\hat{s}:=\frac{d}{s}+\ord_{\infty}(s_{\infty})+\ord_{\infty}(t_{\infty}).
\]
Using the Calabi-Yau condition on the weights, i.e. $w_1+w_2+w_3=0$, one easily computes that
\begin{equation}
-W_2(-p-\hat{r})=W_1(p).
\end{equation}
Therefore we can write
\begin{equation}\label{eqn:cancel}
e^{eq}\left( R^{\bullet}\pi_*f^*N_{\proj_{r,s}^{\mathbf{(L_i,n_i)}}/\Y}\right)=(-1)^{N+1}\prod W_1(p)
\end{equation}
where the product is indexed by
\[
0\longrightarrow M \hspace{1cm}\text{and}\hspace{1cm} -\hat{r}-N\longrightarrow -\hat{r}.
\]
Recalling the index convention of Remark \ref{convention} and using the fact that $M+N+\hat{r}+\hat{s}=0$ (this depends on the CY condition), we see that many of the terms in (\ref{eqn:cancel}) cancel and the remaining terms are indexed by
\[
0\longrightarrow -\hat{r} \hspace{1cm}\text{and}\hspace{1cm} M+\hat{s}\longrightarrow M.
\]
Making the appropriate cancellations and inverting, we compute
\begin{align*}
e^{eq}\left( -R^{\bullet}\pi_*f^*N_{\proj_{r,s}^{\mathbf{(L_i,n_i)}}/\Y}\right)&=(-1)^{N+1}\prod_{p=-\hat{r}+1}^{-1} W_1(p) \prod_{p=M+1}^{M+\hat{s}-1}W_1(p)\\
&=(-1)^{N+1}\prod_{p=1}^{\hat{r}-1} W_1(-p) \prod_{p=1}^{\hat{s}-1}W_1(p+M)\\
\end{align*}
Using the definition of $W_1$, we conclude that the numerator of (\ref{eqn:footballcontrib}) is 
\begin{align}\label{eqn:normalcontrib}
e^{eq}\left( -R^{\bullet}\pi_*f^*N_{\proj_{r,s}^{\mathbf{(L_i,n_i)}}/\Y}\right)=(-1)^{N+1}&\left( \frac{ rw_1}{d} \right)^{\hat{r}-1}\frac{\Gamma\left( \frac{dw_2}{rw_1}+\ord_0(t_0)+\frac{d}{r} \right)}{\Gamma\left( \frac{dw_2}{rw_1} -\ord_0(s_0)+1 \right)}\notag\\
&\cdot\left( \frac{ sw_1'}{d} \right)^{\hat{s}-1}\frac{\Gamma\left( \frac{dw_3'}{sw_1'}+\ord_{\infty}(t_{\infty})+\frac{d}{s} \right)}{\Gamma\left( \frac{dw_3'}{sw_1'} -\ord_{\infty}(s_{\infty})+1 \right)}.
\end{align}

\subsubsection{Denominator of (\ref{eqn:footballcontrib})}

Lastly, we must compute $e^{eq}\left(R^0\pi_*f^*T_{\proj_{r,s}^{\mathbf{(L_i,n_i)}}}\right)$.  Since $T_{\proj_{r,s}^{\mathbf{(L_i,n_i)}}}=\so(1,1,0;0,...,0)$, then $f^*T_{\proj_{r,s}^{\mathbf{(L_i,n_i)}}}=\so(d,d,0).$  If $v_0$ is a minimally vanishing section at $0$, then $H^0(\F,f^*T_{\proj_{r,s}^{\mathbf{(L_i,n_i)}}})$ is generated by
\[
\left\{ v_0z^j\frac{\partial}{\partial z} \right\}_{j=0}^{\lfloor \frac{d}{r} \rfloor + \lfloor \frac{d}{s} \rfloor}
\]
where $z$ is acted on with weight $-\frac{rw_1}{d}$, $v_0$ is acted on with weight $-\frac{rw_1}{d}\left\langle \frac{d}{r} \right\rangle$ and the bundle is linearized with weights $w_1,w_1'$.  Therefore the denominator of (\ref{eqn:footballcontrib}) is
\begin{align}\label{eqn:tangentcontrib}
e^{eq}\left(R^0\pi_*f^*T_{\proj_{r,s}^{\mathbf{(L_i,n_i)}}}\right)&=\prod_{j=0}^{\lfloor \frac{d}{r} \rfloor-1}\prod_{j=\lfloor \frac{d}{r} \rfloor+1}^{\lfloor \frac{d}{s} \rfloor}\left(  w_1-\frac{ rw_1}{d}\left( \left\langle \frac{d}{r} \right\rangle+j\right)\right)\notag\\
&=\left( \frac{ rw_1}{d} \right)^{\lfloor \frac{d}{r} \rfloor}\left\lfloor \frac{d}{r} \right\rfloor !\cdot\left( \frac{ sw_1'}{d} \right)^{\lfloor \frac{d}{s} \rfloor}\left\lfloor \frac{d}{s} \right\rfloor !.
\end{align}

Computing (\ref{eqn:footballcontrib}) with equations (\ref{eqn:normalcontrib}) and (\ref{eqn:tangentcontrib}) shows that the edge contribution from $f$ towards the closed GW invariants is equal to
\[
dn(-1)^{N+1}D_{\vec{k}^0}^+(d;\vec{w})D_{\vec{k}^{\infty}}^-(d;\vec{w}')
\]
which completes the proof.

\section{Insertions}\label{sec:insertions}

In the original formulation of the topological vertex (\cite{akmv:tv}), insertions of cohomology classes were disregarded.  Indeed, for smooth toric CY 3-folds the moduli spaces of stable maps have virtual dimension 0.  Due to the fundamental class axiom, only finitely many invariants with fundamental class insertions are nonzero.  By dimensional reasons, the rest of the invariants have only divisor insertions, which are easily handled by the divisor equation.  In the orbifold case, however, the divisor equation does not hold for the twisted classes in degree $2$ and insertions of these classes tend to give interesting invariants.  Consequently, we develop our formalism to include insertions.

Throughout this section, we restrict to the case where $\Y$ is an \textit{effective} orbifold (i.e. the isotropy of the generic point is trivial).  By the Gorenstein condition, this implies that the nontrivial isotropy is supported in codimension $2$.  One could presumably carry out the computations for non-effective orbifolds, however we do not pursue that here.

\subsection{Orbifold Comology of $\Y$}\label{sec:classes}

To compute the primary insertion invariants, we must first understand the Chen-Ruan orbifold cohomology of $\Y$.  Since we have restricted to the effective case, the only class in degree $0$ is the untwisted fundamental class.  As in the smooth case, if we disregard the finitely many nonzero invariants with fundamental class insertions, then by dimensional reasons the rest of the insertions must be in $H_{CR}^2(\Y)$.

One computes that $H_{CR}^2(\Y)$ is generated by the following classes.
\begin{itemize}
\item Divisor classes $D_e$ in the untwisted sector.
\item Twisted line classes $\mathfrak{c}_{e,h}:=(C_e,h)\subset\mathcal{IY}$ with $h\in G_e$.
\item Twisted point classes $\mathfrak{v}_{i,h}:=(y_i,h)\subset\mathcal{IY}$ where $\Y_i$ can be identified with $\left[\C^3/G_i\right]$, the fixed point set of $h\in G_i$ is $\{y_i\}$, and $h$ acts on $\C^3$ with weights $e^{2\pi i r_j}$ with $\sum r_j=1$.
\end{itemize}

Assign formal variables $t_e, c_{e,h}, v_{i,h}$ corresponding to insertions of these cohomology classes.

\subsection{Orbifold Vertex with Insertions}

Suppose $\X$ is as in section \ref{sec:vertex} such that the action of $G$ is effective.  Let $e_1, e_2, e_3$ be the edges and let $v$ be the vertex in $\Gamma_{\X}$.  As before, $H_{CR}^2(\X)$ is generated by classes $\mathfrak{c}_{e_j,h}$ and $\mathfrak{v}_h$ with corresponding formal variables $c_{e_j,h}$ and $v_h$.  We modify Definition \ref{def:vertex} to include insertions as follows.

\begin{itemize}
\item We denote the invariant with $m_{j,h}$ insertions of $\mathfrak{c}_{e_j,h}$ and $n_h$ insertions of $\mathfrak{v}_h$ by $V_{\X,g,\vec{\mu}}\left((\mathfrak{c}_{e_j,h})^{m_{j,h}}\cdot (\mathfrak{v}_h)^{n_h};\vec{w}\right)$.  The effect of adding these twisted insertions on the localized contribution is simply to prescribe the twisting at the marked points.  Therefore, the corresponding vertex is defined by replacing the moduli space in Definition \ref{def:vertex} with
\[
\M_{g,n+\sum_h(n_h+\sum_j m_{j,h})}\left(\B G\right)\cap ev^*\left(\vec{k}\right)\cap \prod_{h\in G} ev^*(\fc_h)^{n_h+\sum_j m_{j,h}}
\] 
\item $V_\X(\lambda,\mathbf{p},\mathbf{c},\mathbf{v};\vec{w})$ is defined by taking the sum in (\ref{def:vertex2}) over all $g,\vec{\mu}$ as well as all possible insertions of $\mathfrak{c}_{e_j,h}$ and $\mathfrak{v}_h$ and including the formal variables 
\[
\frac{(c_{e_j,h})^{m_{j,h}}(v_h)^{n_h}}{m_{j,h}!n_h!}.
\] 
\item The \textit{oriented GW orbifold vertex with insertions} $V_{\X,\vec\mu}^{\bullet}(\lambda,\mathbf{c},\mathbf{v};\vec{w})$ is defined to be the coefficient of $\mathbf{p}_{\vec\mu}$ in
\[
V_\X^{\bullet}(\lambda,\mathbf{p},\mathbf{c},\mathbf{v};\vec{w}):=\exp\left( V_\X(\lambda,\mathbf{p},\mathbf{c},\mathbf{v};\vec{w}) \right).
\]
\end{itemize}

\subsection{Gluing with Insertions}
In order to obtain the full GW potential of $\Y$ (sans fundamental class insertions), we modify the gluing algorithm of Theorem \ref{gluing} as follows.
\begin{itemize}
\item At each vertex $v_i$, we compute $V_{\Y_i,\Lambda_i}^\bullet(\lambda,\mathbf{c}(i),\mathbf{v}(i);\vec{w}(i))$.
\item We replace $Q_e$ with $\exp(t_e)Q_e$ to account for the divisor equation imposing the relation $t_e=t_{e'}$ if $D_e=D_{e'}\in H^2(\Y)$.
\end{itemize}

\section{Examples}\label{sec:examples}
\subsection{Smooth Case}\label{sec:smooth}
In the smooth case, we recover the computations of \cite{df:lgta} which can then be related to the topological vertex of \cite{akmv:tv} and the generating functions $P_{\vec{\mu}}$ of \cite{orv:qcycc}.  We summarize these correspondences in this section.

Assume that $G$ is trivial so that $\X=\C^3$.  Orient $\Gamma_{\X}$ with all three edges directed outward.  $\vec\mu$ is simply a triple of partitions $(\mu^1,\mu^2,\mu^3)$ where $\mu^i=(d_1^i,...,d_{l_i}^i)$ determines the winding profile along the $i$th Lagrangian.  From Definition \ref{def:vertex}, we compute 
\[
V_{g,\vec\mu}(\vec w)=\frac{\prod_{i,j}w_1w_2w_3D^+(d_j^i;\vec w)}{|\aut(\vec\mu)|}\int_{\M_{g,l(\vec\mu)}}\frac{e^{eq}\left( \mathbb{E}_{g}^{\vee}(w_1)\oplus\mathbb{E}_{g}^{\vee}(w_2)\oplus\mathbb{E}_{g}^{\vee}(w_3)  \right)}{w_1w_2w_3\prod_{i,j}
\left( \frac{w_i}{d_j^i}-\psi_j^i \right)}
\]
where
\[
D^+(d_j^i;\vec w)=\frac{1}{w_id!}\frac{\Gamma\left( \frac{dw_{i+1}}{w_i} + d \right)}{\Gamma\left( \frac{dw_{i+1}}{w_i} +1 \right)}.
\]
Simplifying, we get
\[
\hspace{-.7cm}V_{g,\vec\mu}(\vec w)=\frac{1}{|\aut(\vec{\mu})|}\left[\prod_{i=1}^3\prod_{j=1}^{l_i}\frac{\prod_{k=1}^{d_j^i-1}(d_j^iw_{i+1}+kw_i)}{(d_j^i-1)!w_i^{d_j^i-1}}\right]\int_{\overline{\mathcal{M}}_{g,\l(\vec{\mu})}}\prod_{i=1}^3\frac{\Lambda_g^{\vee}(w_i)w_i^{l(\vec{\mu})-1}}{\prod_{j=1}^{l_i}(w_i(w_i-d_j^i\psi_j^i))}
\]
where $\Lambda_g^{\vee}(w_i):=(-1)^g\sum_{i=0}^g(-w_i)^i\lambda_{g-i}$.
This computation was made in Appendix A of \cite{df:lgta}.  The results of \cite{lllz:mttv} and \cite{moop:gwdtc} together imply that 
\begin{equation}\label{eqn:smvert}
(-\sqrt{-1})^{l(\vec\mu)}V_{\vec\mu}^{\bullet}(\lambda;\vec{w})=C_{\mu^1,\mu^2,\mu^3}^{(\frac{w_2}{w_1},\frac{w_3}{w_2},\frac{w_1}{w_3})}(q)
\end{equation}
where the right side is the framed topological vertex defined in \cite{akmv:tv} via large N duality and where $q=e^{\sqrt{-1}\lambda}$.

\subsubsection{Connection with 3d Partitions}
In the rest of this section, we extract from the literature the explicit connection between (\ref{eqn:smvert}) and the generating functions of 3d partitions defined in \cite{orv:qcycc}.

There are two natural bases for the center of the group ring of $S_d$, related by the character table.  Equation (\ref{eqn:smvert}) is written in terms of the partition basis.  The authors of \cite{akmv:tv} suggest that it is more natural at times to view the vertex in the representation basis.  Also in \cite{akmv:tv}, a particular formula is derived for relating different framings of the vertex.  Applying this change of basis and framing dependency formula to equation (\ref{eqn:smvert}), we compute that the topological vertex at the canonical framing in the representation basis is given by
\begin{equation}\label{eqn:topvert}
C_{\vec\nu}(q)=\left(e^{\sqrt{-1}\lambda}\right)^{\frac{-1}{2}\left(\sum_{i=1}^3\kappa(\nu^i)\frac{w_{i+1}}{w_i}\right)}\sum_{|\mu^i|=|\nu^i|}(-\sqrt{-1})^{l(\vec\mu)}V_{\vec\mu}^{\bullet}(\lambda;\vec{w})\prod_{i=1}^3\chi_{\nu^i}(\mu^i).
\end{equation}
where $\kappa(\nu)=2\sum_{(i,j)\in\nu}(j-i)$, $\chi_\nu$ is the character of $S_{|\nu|}$ indexed by $\nu$, and $q=e^{\sqrt{-1}\lambda}$ (c.f. Proposition 6.6 of \cite{lllz:mttv}).  


In \cite{orv:qcycc}, generating functions $P_{\vec\nu}(q)$ were defined by enumerating 3d partitions with prescribed asymptotics.  They proved the following identity relating their generating functions to the topological vertex in the representation basis at canonical framing.
\begin{equation}\label{eqn:normvert}
\frac{P_{\vec\nu}(q)}{M(q)}=q^{-\frac{1}{2}||\vec\nu^t||^2}C_{\vec\nu}(1/q)
\end{equation}
where $M(q)=\prod_{k\geq 1}(1-q^k)^{-k}$ is the classical MacMahon function and $||\vec\nu||^2:=\sum(\nu_i^j)^2$.  Putting  together equations (\ref{eqn:topvert}) and (\ref{eqn:normvert}), we get the identity
\begin{equation}\label{eqn:mainvert}
\frac{P_{\vec\nu}(q)}{M(q)}=\left( e^{\sqrt{-1}\lambda} \right)^{\frac{1}{2}\left(\sum_{i=1}^3\kappa(\nu^i)\frac{w_{i+1}}{w_i}-||(\nu^i)^t||^2\right)}\sum_{|\mu^i|=|\nu^i|}(\sqrt{-1})^{l(\vec\mu)}V_{\vec\mu}^{\bullet}(\lambda;\vec{w})\prod_{i=1}^3\chi_{\nu^i}(\mu^i).
\end{equation}

\begin{remark}
Our formalism for the orbifold vertex generalizes $V_{\vec\mu}^{\bullet}(\lambda;\vec{w})$ whereas the orbifold vertex formalism of Bryan, Cadman, and Young generalizes $\frac{P_{\vec\nu}(q)}{M(q)}$.  A relation between the two formalisms (the GW/DT correspondence for toric CY orbifolds) should generalize equation (\ref{eqn:mainvert}).
\end{remark}


\subsection{$A_1$ Singularity}\label{sec:a1all}

In this section, we consider the one leg orbifold vertex of $\X=[\C^3/\Z_2]$. It is more convenient in what follows to write $\Z_2=\{1,-1\}$ multiplicatively, we warn the reader that this differs from the conventions used thus far.  We consider both the case where the disks map to an effective or the ineffective leg.  We compute expressions for the vertex using specialized choices of weights.

\subsubsection{Effective Leg}\label{sec:a1ef}

We first consider open invariants for which the only nonempty partition $\mu$ is along an effective leg.  Equip $\X$ with a torus action with weights $(w_1,w_2,w_3)=(1,-1,0)$ where the ineffective leg is designated by the last coordinate and the nonempty partitions are along the leg designated by the first coordinate, which we orient outward.  Analyzing the restriction of the obstruction theory via the specialized weights $\vec w$, we see that the contribution from connected loci with more than one disk vanish.  Furthermore, since the disks are mapping to an effective leg, the twisting is uniquely determined by the winding.  Therefore, the connected vertex contributions are indexed by $\mu=(d,\bar d)$ where $\bar d:=(-1)^d\in\Z_2$.  From Definition \ref{def:vertex}, we compute the contribution from connected loci with genus $g$ and $n$ twisted marked points to be
\[
V_{\X,g,n,(d,\bar d)}(\vec w)=\delta_{\bar n, \bar d}2(-1)^{d-1}\left(\frac{d}{2} \right)^{\delta_{1,\bar d}}\hspace{-.3cm}D_{\bar d}^+(d,\vec w)\int_{\M_{g;n,1}(\B\Z_2)\cap ev_{n+1}^*(\bar d)}\hspace{-1.5cm}\frac{e^{eq}\left( \mathbb{E}_{-1}^{\vee}(1)\oplus\mathbb{E}_{-1}^{\vee}(-1)\oplus\mathbb{E}_{1}^{\vee}(0) \right)}{\left(\frac{2}{d}-\psi\right)}
\]
where the disk function simplifies
\[
D_{\bar d}^+(d,\vec w)=\frac{(-1)^{\lfloor \frac{d-1}{2} \rfloor}}{d}\left(\frac{2}{d}\right)^{\delta_{1,\bar d}}.
\]
$\M_{g;n,1}(\B\Z_2)$ denotes the locus of $n+1$ stable maps where the first $n$ points are twisted, $ev_i^*(\pm 1)$ denotes pullback of $\fc_{\pm 1}$, and the term $\delta_{\bar n, \bar d}$ reflects the fact that the moduli space only exists if $n$ and $d$ have the same parity.  The integral simplifies
\begin{align*}
\int_{\M_{g;n,1}(\B\Z_2)\cap ev_{n+1}^*(\bar d)}&\hspace{-1cm}\frac{e^{eq}\left( \mathbb{E}_{-1}^{\vee}(1)\oplus\mathbb{E}_{-1}^{\vee}(-1)\oplus\mathbb{E}_{1}^{\vee}(0) \right)}{\left(\frac{2}{d}-\psi\right)}=
\int_{\M_{g;n,1}(\B\Z_2)\cap ev_{n+1}^*(\bar d)}\frac{(-1)^{\lfloor\frac{n-1}{2}\rfloor}\lambda_g}{\prod \left(\frac{2}{d}-\psi\right)}\\
&=(-1)^{\lfloor\frac{n-1}{2}\rfloor}\left( \frac{d}{2} \right)^{2g-1+n}\int_{\M_{g;n,1}(\B\Z_2)\cap ev_{n+1}^*(\bar d)}\lambda_g\psi^{2g-2+n}\\
&=(-1)^{\lfloor\frac{n-1}{2}\rfloor}\left( \frac{d}{2} \right)^{2g-1+n}2^{2g-1}\int_{\M_{g,n,1}}\lambda_g\psi^{2g-2+n}\\
&=\frac{(-1)^{\lfloor\frac{n-1}{2}\rfloor}d^{2g-1+n}}{2^n}b_g
\end{align*}
where the first equality uses the $\Z_2$-Mumford relation (\cite{bgp:crc}, Proposition 3.2), the second uses the fact that the map $\M_{g;n,1}(\B\Z_2)\cap ev_{n+1}^*(\bar d)\rightarrow \M_{g,n}$ has degree $2^{2g-1}$ if $n$ and $\bar d$ have the same parity (\cite{jpt:ahhi}, Lemma 6), and we define $b_g:=\int_{\M_{g,1}}\lambda_g\psi^{2g-2}$.  Therefore,

\[
V_{\X,g,n,(d,\bar d)}(\vec w)=\delta_{\bar n, \bar d}\frac{(-1)^{\lfloor \frac{d-1}{2} \rfloor+\lfloor\frac{n}{2}\rfloor}d^{2g-2+n}}{2^{n-1}}b_g
\]
Inserting the formal variables $\lambda^{2g-1}\frac{x^n}{n!}$, summing over all $g$ and $n$, and using the known generating function for $b_g$ (\cite{fp:hiagwt}), we compute
\begin{equation}\label{eqn:z2vertef}
V_{\X,\mu}(x,\lambda;\vec w)=
\end{equation}
\[
\begin{cases}
\frac{-\sqrt{-1}^d}{d}\frac{\cos(dx/2)}{\sin(d\lambda/2)} & \text{ if } \mu=((d,\bar d)) \text{ with } d \text{ even }\\
\frac{-\sqrt{-1}^{d-1}}{d}\frac{\sin(dx/2)}{\sin(d\lambda/2)} & \text{ if } \mu=((d,\bar d)) \text{ with } d \text{ odd }\\
0 & \text{ else}.
\end{cases}
\]

\subsubsection{Ineffective Leg}\label{sec:a1}

We now consider open invariants of $\X=[\C^3/\Z_2]$ for which the only nonempty partition $\mu$ is along the ineffective leg.  Equip $\X$ with a torus action with weights $(w_1,w_2,w_3)=(1,-1,0)$ where the ineffective leg is now designated by the first coordinate and oriented outward.   The genus 0 vertex was computed in \cite{cr:ogwcrc} (Prop 4.2):

\begin{equation}\label{eqn:z2vert}
V_{\X,0,\mu}(x;\vec w)=
\end{equation}
\[
\begin{cases}
\frac{(-1)^{d+1}}{2d^2} & \text{ if } \mu=((d,1))\\
\frac{(-1)^d}{|\text{Aut}(\vec\mu)|}\Bigg(\prod_{i=1}^n\frac{(2d_i-1)!!}{(2d_i)!!}\Bigg)\left(\frac{d^{n-2}}{dx^{n-2}}\frac{\sec^{2d}(x/2)}{2d}\right) & \text{ if } \mu=((d_1,-1),...,(d_n,-1))\\
0 & \text{ else}.
\end{cases}
\]
where $x$ tracks insertions of the twisted axis.  This result follows easily from Definition \ref{def:vertex} and the hyperelliptic Hodge integral identity proven in Section 2.3 of \cite{cr:ogwcrc}.

We can also compute the all genus vertex contribution $V_{\X,(1,-1)}(x;\vec w)$ using the following $\Z_2$-Hodge integral identity which we prove in Appendix \ref{sec:hhi}.

\begin{theorem}\label{hhi}
\[
\hspace{-.2cm}\sum_{g,n}\int_{\overline{\mathcal{M}}_{g;n,0}(\mathcal{B}\Z_2)}\hspace{-1cm}\frac{e^{eq}\left( \mathbb{E}_{1}^{\vee}(1)\oplus\mathbb{E}_{-1}^{\vee}(-1)\oplus\mathbb{E}_{-1}^{\vee}(0) \right)}{1-\psi}\frac{x^{n-1}}{(n-1)!}\lambda^{2g-1}=\frac{1}{2}\csc\left(\frac{\lambda}{2}\right)\tan\left(\frac{x}{2}\right).
\]
\end{theorem}

The remaining terms in the definition of $V_{\X,g,(1,-1)}(\vec w)$ are easily seen to contribute a factor of $-1/2$.  Therefore, Theorem \ref{hhi} implies
\begin{equation}\label{d1twist}
V_{\X,(1,-1)}(x,\lambda;\vec w)=\frac{-1}{4}\csc\left(\frac{\lambda}{2}\right)\tan\left(\frac{x}{2}\right).
\end{equation}

Finally, we compute $V_{\X,(d,1)}(x;\vec w)$.  By analyzing the obstruction theory, we see that untwisted nodes survive the $0$ weight only if they are attached to a contracted component with no twisted marks.  Further, the map from the contracted component to $\B\Z_2$ must correspond to the trivial \'{e}tale $\Z_2$ cover.  Therefore, these fixed loci are isomorphic to $\frac{1}{2}[\M_{g,1}]$ where $g$ is the genus of the contracted component and the factor of $\frac{1}{2}$ corresponds to the hyperelliptic involution.  Furthermore, in this case $\mathbb{E}_1\cong\mathbb{E}_{-1}$.  Applying the Mumford relation to the integrand in Definition \ref{def:vertex} simply leaves the $d^{2g-2}\lambda_g\psi^{2g-2}$.  Putting this together with the rest of the contribution in Definition \ref{def:vertex} we compute
\[
V_{\X,g,(d,1)}(\vec w)=\frac{(-1)^{d-1}d^{2g-2}}{2}b_g
\]
Inserting the formal variable $\lambda^{2g-1}$ and summing over all $g$, we compute
\begin{equation}\label{d1untwist}
V_{\X,(d,1)}(x,\lambda;\vec w)=\frac{(-1)^{d-1}}{4d}\csc\left( \frac{d\lambda}{2} \right).
\end{equation}

We will return to these computations in the following section.

\section{Applications}\label{sec:apps}

\subsection{Gromov-Witten/Donaldson Thomas Correspondence}\label{sec:gwdt}

The GW/DT conjecture first appeared in \cite{mnop:gwdt1} and \cite{mnop:gwdt2}.  Essentially, the conjecture states that the reduced GW and DT theories of a smooth 3-fold are equivalent after a change in formal variables.  In \cite{moop:gwdtc}, the conjecture was proven for smooth toric 3-folds and one consequence is the equivalence of the vertex theories mentioned in Section \ref{sec:smooth}.  In \cite{bcy:otv}, an orbifold GW/DT correspondence was suggested.  In this section, we give evidence that suggests that the orbifold GW/DT correspondence should be witnessed by the orbifold vertex.  In the present work, we content ourselves with the specific formulation of the correspondence for the one-leg $\Z_2$ vertex.  We leave the more general formulation of the orbifold vertex correspondence for future work.

\subsubsection{Effective Leg}\label{ssec:a1ef}

Denote by $V_{\mu}$ the effective one-leg $\Z_2$ vertex computed with the particular weights of Section \ref{sec:a1ef} and $P_{\nu}$ the corresponding reduced, multi-regular DT vertex appearing in Example 4.2 of \cite{bcy:otv}.  Note that $\mu$ and $\nu$ are simply partitions, since the isotropy along the effective leg is trivial.  As usual, we use $\chi_\nu(\mu)$ to denote the value of the character indexed by $\nu$ at the conjugacy class indexed by $\mu$.

\begin{theorem}\label{effectivecor}
After the identifications $q_0=e^{\sqrt{-1}(\lambda-x)}$ and $q_1=-e^{\sqrt{-1}x}$,
\[
P_{\nu}(q_0,q_1)=\left(e^{\sqrt{-1}\lambda}\right)^{B_0(\nu)}\left(-e^{\sqrt{-1}x}\right)^{B_1(\nu)}\sum_{|\mu|=|\nu|}\sqrt{-1}^{l(\mu)}V_{\mu}^{\bullet}(x,\lambda) \chi_{\nu}(\mu)
\]
where $B_0(\nu):=-\frac{|\nu|}{2}-\sum_{(i,j)\in\nu}\lfloor \frac{i}{2} \rfloor$ and $B_1(\nu):=\sum_{(i,j)\in\nu}\lfloor \frac{i+1}{2} \rfloor-B_0(\nu)$.
\end{theorem}

\begin{proof}
From \cite{bcy:otv}, we know
\[
P_{\nu}(-q_0,q_1)=q_0^{B_0(\nu)+\frac{|\nu|}{2}}q_1^{B_1(\nu)+B_0(\nu)}s_{\nu'}(q_{\bullet})
\]
where $s_{\nu'}$ is the Schur function corresponding to the conjugate of $\nu$ and 
\[
q_{\bullet}=(1,q_1,q_0q_1,q_0q_1^2,q_0^2q_1^2,...)=(1,q_1,q,q_1q,q^2,q_1q^2,...)
\]
where $q:=q_0q_1$.
Therefore, we must show
\begin{equation}\label{eqn:schur1}
q_0^{\frac{1}{2}|\nu|}s_{\nu'}(q_{\bullet})=\sum_{|\mu|=|\nu|}\sqrt{-1}^{l(\mu)}V_{\mu}^{\bullet} \chi_{\nu}(\mu)
\end{equation}
after the change of variables $q_0=-e^{\sqrt{-1}(\lambda-x)}$ and $q_1=-e^{\sqrt{-1}x}$.  Define $p_i(\vec y)=y_0^i+y_1^i+y_2^i+...$.  Then (\ref{eqn:schur1}) is equivalent to showing that the disconnected GW vertex is given (after the change of variables) by
\begin{equation}\label{eqn:schur2}
\sum_{\mu\neq\emptyset}\sqrt{-1}^{l(\mu)}V_{\mu}p_\mu(\vec y)=\log\left( \sum_{d\geq 1} \sum_{|\mu|=|\nu|=d} q_0^{\frac{d}{2}}s_{\nu'}(q_{\bullet}) \frac{\chi_{\nu}(\mu)}{z_\mu} p_\mu(\vec y) \right)
\end{equation}
where $z_\mu$ is the order of the centralizer of $\mu$ in $S_d$.  Summing over $\mu$ and using the identity
\[
\sum_{|\mu|=d}\frac{\chi_{\nu}(\mu)}{z_\mu} p_\mu(\vec y)=s_{\nu}(\vec y),
\]
the right hand side of (\ref{eqn:schur2}) becomes
\begin{equation}\label{eqn:schur3}
\log\left( \sum_{d\geq 1} \sum_{|\mu|=d} q_0^{\frac{d}{2}}s_{\nu'}(q_{\bullet})s_{\nu}(\vec y) \right)
\end{equation}

Cauchy's identity implies that (\ref{eqn:schur3}) is equal to
\begin{align}
\log&\left( \prod_{i,j\geq 0}(1+q_0^{\frac{1}{2}}y_iq^j) \prod_{i,j\geq 0}(1+q_0^{\frac{1}{2}}q_1y_iq^j)\right)\\
&=\sum_{i,j\geq 0}\log(1+q_0^{\frac{1}{2}}y_iq^j)+\sum_{i,j\geq 0}\log(1+q_0^{\frac{1}{2}}q_1y_iq^j)\\
&=\sum_{d\geq 1}(-1)^{d-1}\left( \sum_{i,j\geq 0}\frac{(q_0^{\frac{1}{2}})^d}{d}y_i^dq^{jd} +\sum_{i,j\geq 0}\frac{(q_0^{\frac{1}{2}}q_1)^d}{d}y_iq^{jd} \right)\\
&=\sum_{d \geq 1}(-1)^{d-1} \frac{q_0^{\frac{d}{2}}}{d}\frac{1+q_1^d}{1-q^d}p_d(\vec y).\label{eqn:schur4}
\end{align}

Making the prescribed change of variables and simplifying, we see that the coefficient of $p_d(\vec y)$ in (\ref{eqn:schur4}) is equal to
\[
\begin{cases}
\frac{-\sqrt{-1}^{d+1}}{d}\frac{\cos(x/2)}{\sin(d\lambda/2)} & \text{ if } \mu=((d,\bar d)) \text{ with } d \text{ even }\\
\frac{-\sqrt{-1}^{d}}{d}\frac{\sin(x/2)}{\sin(d\lambda/2)} & \text{ if } \mu=((d,\bar d)) \text{ with } d \text{ odd }
\end{cases}
\]
which is exactly $\sqrt{-1}\cdot V_{\mu}$ where $V_{\mu}$ was derived in Section \ref{sec:a1ef}.

\end{proof}

\subsubsection{Ineffective Leg}\label{ssec:a1in}

In this section, denote by $V_{\mu}$ the ineffective one-leg $\Z_2$ vertex computed with the particular weights of Section \ref{sec:a1} and $P_{\nu}$ the corresponding reduced, multi-regular DT vertex appearing in Example 4.3 of \cite{bcy:otv}.  Note that $\mu$ and $\nu$ are not simply partitions, but decorated partitions which we now describe.

Recall that $\mu$ is given by a set of pairs
\[
\mu=\{(d_1,k_1),...,(d_n,k_n)\}
\]
with $k_i\in \Z_2=\{\pm 1\}$.  
We can abbreviate $\mu$ by
\[
\mu=\{d_1k_1,...,d_nk_n\}.
\]
If $d=\sum d_i$, then $\mu$ naturally corresponds to a conjugacy class in the wreath product $\Z_2\wr S_d$ (c.f. \cite{m:sfhp}).  Define the twisted length of $\mu$ by $\hat{l}(\mu):=\sum_{ord(k_i)=j}j$.

In contrast, the contributions to the one-leg DT vertex computed in Example 4.3 of \cite{bcy:otv} are indexed by partitions, colored with 2 colors, satisfying a ``balanced'' condition.  We quickly review these notions here.  Consider a partition $\nu$ of $2d$ represented as a Young diagram in French notation in the first quadrant.  Define $\nu[k]:=\{(i,j)\in\nu:i-j=k \mod 2\}$ (one can think of coloring the boxes with $2$ colors).  Define a \textit{balanced} partition to be one for which $|\nu[0]|=|\nu[1]|$.

Given such a balanced partition $\nu$, we can construct an irreducible representation of $\Z_2\wr S_d$ as follows.  We begin by constructing the $2$-quotient of $\nu$ (c.f. Section 7.3 of \cite{bcy:otv}).  The $2$-core is empty because $\nu$ is balanced.  The $2$-quotient consists of 2 partitions, one labelled by each of the irreducible representations of $\Z_2$.  Concatenating the 2 partitions, we get a partition of $d$, where each part is labelled by an irreducible representation of $\Z_2$.  Following Appendix A of Chapter I of \cite{m:sfhp}, such a labelled partition corresponds exactly to an irreducible representation of $\Z_2\wr S_d$ and we denote the corresponding character by $\chi_{\nu}$.  

These identifications between the discrete data of our one leg orbifold vertices and the representation theory of the wreath product $\Z_2\wr S_d$ allow us to make the following precise conjecture relating the relevant GW and DT vertices.

\begin{conjecture}\label{conjgwdtin}
After the identifications $q_0=e^{\sqrt{-1}(\lambda-x)}$ and $q_1=-e^{\sqrt{-1}x}$,
\[
P_{\nu}(q_0,q_1)=\left(e^{\sqrt{-1}\lambda}\right)^{C_0(\nu)}\left(e^{\sqrt{-1}x}\right)^{C_1(\nu)}\sum_{2|\mu|=|\nu|}(\sqrt{-1})^{\hat{l}\mu}V_{\mu}^{\bullet}(x,\lambda) \chi_{\nu}(\mu).
\]
\hspace{-.1cm}where $C_0(\nu):=-\frac{1}{2}\sum_{(i,j)\in\nu[0]}2j+1$ and $C_1(\nu):=-\frac{1}{2}\sum_{(i,j)\in\nu[1]}2j+1-C_0(\nu)$.
\end{conjecture}

As evidence for Conjecture \ref{conjgwdtin}, we prove the following.

\begin{theorem}\label{gwdtin}
Conjecture \ref{conjgwdtin} is true if
\begin{itemize}
\item $d=1$,
\item $d=2$ and the power of $\lambda$ is nonpositive, or
\item $d=3$ and the power of $\lambda$ is negative.
\end{itemize}
\end{theorem}

\begin{proof}

\noindent\underline{$\mathbf{d=1}$}:

We first consider the degree $1$ case.  The GW parition functions are
\[
V_{(1)}^\bullet=\frac{1}{4}\csc\left(\frac{\lambda}{2}\right),
\]
\[
V_{(-1)}^\bullet=\frac{-1}{4}\csc\left(\frac{\lambda}{2}\right)\tan\left(\frac{x}{2}\right),
\]\\

The DT partition functions are 
\[
P_{(1,1)}'=P_{(2)}'=\frac{1}{(1+q_0q_1)(1-q_1)}
\]\\

\begin{table}[h]\label{table1}
\caption{The character table of $\Z_2\wr S_1\cong \Z_2$.}
\begin{tabular} {c|cc}
&(1)&(-1)\\
\hline
(1,1)&1&1\\
(2)&1&-1\\
\end{tabular}
\end{table}

The relevant character table is that of $\Z_2$ and the identifications of the indices are given in Table 1.
Therefore the conjecture states that
\[
\frac{1}{(1-e^{\sqrt{-1}\lambda})(1+e^{\sqrt{-1}x})}=e^{-\sqrt{-1}\lambda/2}\left( \sqrt{-1}V_{(1)}^{\bullet}-V_{(-1)}^{\bullet} \right)
\]
and
\[
\frac{1}{(1-e^{\sqrt{-1}\lambda})(1+e^{\sqrt{-1}x})}=e^{-\sqrt{-1}(\lambda/2+x)}\left( \sqrt{-1}V_{(1)}^{\bullet}+V_{(-1)}^{\bullet} \right).
\]
These identities can be checked with basic trigonometric identities.  Therefore, the conjecture is proven in degree one.\\

\noindent\underline{$\mathbf{d=2}$}:

In degree two, we do not have the luxury of closed forms for all genus.  Therefore, we must make some approximations for the GW partition functions:
\[
V_{(2)}^{\bullet}=\frac{-1}{8}\csc(\lambda)
\]
\[
V_{(-2)}^{\bullet}=\frac{3}{32}\int\sec^4(x/2)dx\lambda^{-1} + O(\lambda)
\]
\[
V_{(1,1)}^{\bullet}= \frac{1}{32}\csc^2(\lambda/2)
\]
\[
V_{(1,-1)}^{\bullet}= \frac{-1}{16}\csc^2(\lambda/2)\tan(x/2)
\]
\[
V_{(-1,-1)}^{\bullet}=\frac{1}{32}\csc^2(\lambda/2)\tan^2(x/2)+\frac{1}{32}\sec^4(x/2)+O(\lambda^2)
\]\\

The DT partition functions are
\[
P_{(4)}'=P_{(1,1,1,1)}'=\frac{1}{(1-q_0^2q_1^2)(1+q_0q_1^2)(1+q_0q_1)(1-q_1)}
\]
\[
P_{(3,1)}'=P_{(2,1,1)}'=\frac{1}{(1-q_0^2q_1^2)(1+q_0q_1)(1+q_0)(1-q_1)}
\]
\[
P_{(2,2)}'=\frac{1}{(1+q_0q_1^2)(1+q_0q_1)^2(1+q_0)}
\]\\

\begin{table}[h]\label{table2}
\caption{The character table of $\Z_2\wr S_2$.}
\begin{tabular} {c|ccccc}
&(2)&(-2)&(1,1)&(1,-1)&(-1,-1)\\
\hline
(3,1)&1&1&1&1&1\\
(4)&1&-1&1&-1&1\\
(1,1,1,1)&-1&-1&1&1&1\\
(2,2)&0&0&2&0&-2\\
(2,1,1)&-1&1&1&-1&1
\end{tabular}
\end{table}

The character table of $\Z_2\wr S_2$ along with the index identifications is given in Table 2.  Using this data it is not hard to check the conjecture in degree 2 for the coefficients of $\lambda^i$ with $i<1$.  A Maple file with these computations can be found at www.math.colostate.edu/\raise.17ex\hbox{$\scriptstyle\mathtt{\sim}$}ross/research.\\

\noindent\underline{$\mathbf{d=3}$}:

As one last bit of evidence, we ran the check in degree three.  The GW partition functions are
\[
V_{(3)}^\bullet=\frac{1}{12}\csc(3\lambda/2)
\]
\[
V_{(-3)}^\bullet=\frac{-5}{96}\int \sec^6(x/2)dx \lambda^{-1}+O(\lambda)
\]
\[
V_{(2,1)}^\bullet=\frac{-1}{32}\csc(\lambda)\csc(\lambda/2)
\]
\[
V_{(2,-1)}^\bullet=\frac{1}{32}\csc(\lambda)\csc(\lambda/2)\tan(x/2)
\]
\[
V_{(-2,1)}^\bullet=\frac{3}{128}\int\sec^4(x/2)dx\cdot\csc(\lambda/2)\lambda^{-1}+O(\lambda^0)
\]
\[
V_{(-2,-1)}^\bullet=\frac{-3}{128}\int\sec^4(x/2)dx\cdot\csc(\lambda/2)\tan(x/2)\lambda^{-1}+\frac{-1}{32}\sec^6(x/2)+O(\lambda^0)
\]
\[
V_{(1,1,1)}^\bullet=\frac{1}{384}\csc^3(\lambda/2)
\]
\[
V_{(1,1,-1)}^\bullet=\frac{-1}{128}\csc^3(\lambda/2)\tan(x/2)
\]
\[
V_{(1,-1,-1)}^\bullet=\frac{1}{128}\csc^3(\lambda/2)\tan^2(x/2)+\frac{1}{128}\csc(\lambda/2)\sec^4(x/2)+O(\lambda)
\]
\[
\hspace{-1.5cm}V_{(-1,-1,-1)}^\bullet=\frac{-1}{384}\csc^3(\lambda/2)\tan^3(x/2)+\frac{-1}{128}\csc(\lambda/2)\tan(x/2)\sec^4(x/2)+\frac{-1}{288}\sec^6(x/2)'\lambda+O(\lambda)
\]\\

The DT partition functions are
\[
P_{(6)}'=P_{(1,1,1,1,1,1)}'=\frac{1}{(1+q_0^3q_1^3)(1-q_0^2q_1^3)(1-q_0^2q_1^2)(1+q_0q_1^2)(1+q_0q_1)(1-q_1)}
\]
\[
P_{(5,1)}'=P_{(2,1,1,1,1)}'=\frac{1}{(1+q_0^3q_1^3)(1-q_0^2q_1^2)(1-q_0^2q_1)(1+q_0q_1)(1+q_0)(1-q_1)}
\]
\[
P_{(4,2)}'=P_{(2,2,1,1)}'=\frac{1}{(1-q_0^2q_1^3)(1-q_0^2q_1^2)(1+q_0q_1)^2(1+q_0)(1-q_1)}
\]
\[
P_{(4,1,1)}'=P_{(3,1,1,1)}'=\frac{1}{(1+q_0^3q_1^3)(1+q_0q_1^2)(1+q_0q_1)^2(1+q_0)(1-q_1)}
\]
\[
P_{(3,3)}'=P_{(2,2,2)}'=\frac{1}{(1-q_0^2q_1^2)(1-q_0^2q_1)(1+q_0q_1^2)(1+q_0q_1)^2(1-q_1)}
\]\\

\begin{table}[!]\label{table3}
\caption{The character table of $\Z_2\wr S_3$.}
\resizebox{12.5cm}{!}{
\begin{tabular} {c|cccccccccc}
&(3)&(-3)&(2,1)&(2,-1)&(-2,1)&(-2,-1)&(1,1,1)&(1,1,-1)&(1,-1,-1)&(-1,-1,-1)\\
\hline
(5,1)& 1& 1& 1& 1& 1& 1& 1& 1& 1& 1\\
(6)& 1&-1& 1&-1&-1& 1& 1&-1& 1&-1\\ 
(3,1,1,1)&-1&-1& 0& 0& 0& 0& 2& 2& 2& 2\\
(3,3)& 0& 0& 1&-1& 1&-1& 3& 1&-1&-3\\
(4,2)& 0& 0& 1& 1&-1&-1& 3&-1&-1& 3\\
(4,1,1)&-1& 1& 0& 0& 0& 0& 2&-2& 2&-2\\
(1,1,1,1,1,1)& 1& 1&-1&-1&-1&-1& 1& 1& 1& 1\\
(2,2,1,1)& 0& 0&-1& 1&-1& 1& 3& 1&-1&-3\\
(2,2,2)& 0& 0&-1&-1& 1& 1& 3&-1&-1& 3\\
(2,1,1,1,1)& 1&-1&-1& 1& 1&-1& 1&-1& 1&-1
\end{tabular}}
\end{table}

The character table of $\Z_2\wr S_3$ along with the index identifications is given in Table 3.  Using this data it is not hard to check the conjecture in degree 3 for the coefficients of $\lambda^i$ with $i<0$.  A Maple document with these computations can be found at www.math.colostate.edu/\raise.17ex\hbox{$\scriptstyle\mathtt{\sim}$}ross/research.

\end{proof}

\subsection{Predictions of Higher Genus $\Z_2$ Hodge Integrals}\label{sec:pred}

Hodge integrals are notoriously elusive to compute and package into generating functions.  G-Hodge integrals pose even more challenges.  One further application of the GW/DT vertex correspondence is the prediction of generating functions for certain G-Hodge integrals.  Ideally, one could verify the conjecture in genus $0$ where the integrals are often tractable then use the correspondence to compute the higher genus integrals.  In case $\X=[\C^3/\Z_2]$, we ran this program in degree $2$.

Define

\[
G(i,g):=\sum_{n}\int_{\overline{\mathcal{M}}_{g;n,0}(\mathcal{B}\Z_2)}\frac{e^{eq}\left( \mathbb{E}_{1}^{\vee}(1)\oplus\mathbb{E}_{-1}^{\vee}(-1)\oplus\mathbb{E}_{-1}^{\vee}(0) \right)}{\prod_{j=1}^i\frac{1}{d_j}-\psi_j}\frac{x^{n-i}}{(n-i)!}.
\]

We predict $G(i,g)$ for values $i=1, d_1=2$ and $i=2, d_1=d_2=1$ with genus $g$ ranging from $1$ to $3$.  To compute $G(1,1)$ for example, we consider the degree $2$, $\lambda^{1}$ part of Conjecture \ref{conjgwdtin}.  All of the contributions to the conjectural equation are known except for the contributions from a contracted genus $1$ curve with two attached disks.  The $\Z_2$-Hodge integrals that appear in these particular contributions are packaged exactly as in $G(1,1)$, therefore we can solve for a conjectural expression for the integrals.  Using analagous methods for the coefficients of different powers of $\lambda$, we compute the following.

\[
G(1,1)=\int\frac{-1}{12}\sec^4\left(\frac{x}{2}\right)+\frac{5}{24}\sec^6\left(\frac{x}{2}\right)dx\\
\]
\[
G(1,2)=\int\frac{-1}{240}\sec^4\left(\frac{x}{2}\right)-\frac{13}{288}\sec^6\left(\frac{x}{2}\right)+\frac{7}{96}\sec^8\left(\frac{x}{2}\right)dx\\
\]
\[
G(1,3)=\int\frac{-11}{30240}\sec^4\left(\frac{x}{2}\right)+\frac{1}{576}\sec^6\left(\frac{x}{2}\right)-\frac{1}{48}\sec^8\left(\frac{x}{2}\right)+\frac{3}{128}\sec^{10}\left(\frac{x}{2}\right)dx
\]
\[
G(2,1)=\frac{1}{16}\sec^6\left(\frac{x}{2}\right)\\
\]
\[
G(2,2)=\frac{-1}{192}\sec^6\left(\frac{x}{2}\right)+\frac{1}{64}\sec^8\left(\frac{x}{2}\right)\\
\]
\[
G(2,3)=\frac{1}{5760}\sec^6\left(\frac{x}{2}\right)-\frac{1}{384}\sec^8\left(\frac{x}{2}\right)+\frac{1}{256}\sec^{10}\left(\frac{x}{2}\right)\\
\]

\appendix
\section{Localization and Higher Genus $\Z_2$ Hodge Integrals}\label{sec:hhi}

We prove the identity of Theorem \ref{hhi}.  For $n\geq 2$, define

\[
F_{g,n}:=\int_{\overline{\mathcal{M}}_{g;n,0}(\mathcal{B}\Z_2)}\frac{e^{eq}\left( \mathbb{E}_{1}^{\vee}(1)\oplus\mathbb{E}_{-1}^{\vee}(-1)\oplus\mathbb{E}_{-1}^{\vee}(0) \right)}{1-\psi}.
\]
We use the convention that $F_{0,2}=\frac{1}{2}$.  We show:

\begin{equation}\label{ident}
\sum_{g,n} F_{g,n}\frac{x^{n-1}}{(n-1)!}\lambda^{2g-1}=\frac{1}{2}\csc\left(\frac{\lambda}{2}\right)\tan\left(\frac{x}{2}\right).
\end{equation}

\subsection{Localization Setup}

We will prove (\ref{ident}) by applying Attiyah-Bott localization to the following auxilary integral:

\begin{equation}\label{int}
I_{g,n}:=\int_{\M_{g;n-1,0}(\mathcal{G},1)}e(-R^{\bullet}\pi_*f^*(\so(-1/2)\oplus\so(-1/2)))
\end{equation}
where $\G$ is the nontrivial $\Z_2$ gerbe over $\proj^1$, i.e. $\G=\proj^{(\so(-1),2)}$ in the notation of Section \ref{sec:orbiline} and $\so(-1/2)=\sqrt{\so(-1)}$.

The integral $I_{g,n}$ vanishes by dimensional reasons.  Indeed the dimension of the moduli space is $2g+n-1$ whereas the orbifold Riemann-Roch formula implies that
\[
\chi(f^*(\so(-1/2)))=(1-g)-1/2-\frac{n-1}{2}
\]
showing that the degree of the integrand is $2g+n-2$.

We equip the target with a $\C^*$ action with the following choices of linearization
\begin{center}
\begin{tabular}{c|ccc}
&$T_{\G}$&$\so(-1/2)$&$\so(-1/2)$\\
\hline
0&1&-1&0\\
$\infty$&-1&-1/2&1/2
\end{tabular}
\end{center}

We can write the integral as a sum of contributions coming from the fixed loci.  The fixed loci are represented by maps to $\G$ where the source curve is the union of a teardrop (mapping with degree one) and two contracted components over $0$ and $\infty$ which support all of the marks and the genus.  In the next section, we show that many of the contributions from the fixed loci are zero.

\subsection{Vanishing Lemmas}

\begin{lemma}\label{vanish1}
If a node over $\infty$ is untwisted, then the contribution vanishes unless the contraction over $\infty$ parametrizes the trivial double cover.  In particular, all of the twisted marks must map to $\infty$.
\end{lemma}

\begin{proof}
Analyzing the normalization sequence, one sees that the skyscraper sheaf at the untwisted node has a $\Z_2$ anti-invariant section, and thus contributes a $0$ weight.  The only way this $0$ weight can be cancelled is if the contracting curve also has a $\Z_2$ anti-invariant section and this occurs only if the contraction map parametrizes the trivial double cover.
\end{proof}



\begin{remark}\label{remark}
By Lemma \ref{vanish1}, the only nonzero contributions with untwisted nodes over $0$ come from maps where the contraction over $0$ parametrizes the trivial \'etale double cover of the contracting component.  We denote this locus by $\mathcal{T}$.  The \textit{coarse} locus T is naturally isomorphic to $\overline{M}_{h,1}$ where $h$ is the genus of the curve contracted to $0$.  However, the natural map of stacks $\rho:\mathcal{T}\rightarrow\M_{h,1}$ is degree $1/2$ because it forgets the double cover along with the nontrivial automorphism of the double cover.
\end{remark}

\begin{lemma}\label{vanish2}
If a component contracted to $\infty$ has positive genus and the node attaching it to the teardrop is twisted, then the contribution vanishes.
\end{lemma}

\begin{proof}
Consider a fixed locus with genus splitting $g_1+g_2=g$ and twisted node over $\infty$.  Then the node at $0$ must be untwisted, therefore by Lemma \ref{vanish1} all of the marks must be supported over $\infty$.  The contribution to (\ref{int}) from such a locus is given by
\[
\hspace{-1cm}-\int_{\mathcal{T}}\frac{e^{eq}\left( \mathbb{E}_{1}^{\vee}(1)\oplus\mathbb{E}_{-1}^{\vee}(-1)\oplus\mathbb{E}_{-1}^{\vee}(0) \right)}{1-\psi}\cdot\int_{\M_{g_2;n,0}(\B\Z_2)}\hspace{-.5cm}\frac{e^{eq}\left( \mathbb{E}_{1}^{\vee}(-1)\oplus\mathbb{E}_{-1}^{\vee}(-1/2)\oplus\mathbb{E}_{-1}^{\vee}(1/2) \right)}{-1-\psi}
\]

By the $\Z_2$ Mumford relation (\cite{bgp:crc}, Proposition 3.2), the second integral simplifies to
\begin{equation}\label{int3}
\frac{(-1)^{g_2-1+n/2}}{2^{2g_2-2+n}}\int_{\M_{g_2;n,0}(\B\Z_2)}\frac{e^{eq}\left( \mathbb{E}_{1}^{\vee}(-1) \right)}{-1-\psi}
\end{equation}

We show that this last integral vanishes if $g_2>0$.  To do this, consider the auxilary integral
\begin{equation}\label{int2}
J_{g_2,n}:=\int_{\M_{g_2;n-1,0}(\mathcal{G},1)}\prod_{i=1}^{n-1}ev_i^*\so(1).
\end{equation}
The integral vanishes by dimensional reasons so long as $g_2>0$.  Linearizing the bundles with weights $0$ at $0$ and $-1$ at $\infty$ forces all of the marks to $\infty$.  Therefore, the fixed loci are indexed by all the ways to split the genus $h_1+h_2=g_2$.  The contribution from such a fixed locus is
\[
-\int_{\M_{h_1;0,1}(\B\Z_2)}\frac{e^{eq}\left( \mathbb{E}_{1}^{\vee}(1) \right)}{1-\psi}\cdot\int_{\M_{h_2;n,0}(\B\Z_2)}\frac{e^{eq}\left( \mathbb{E}_{1}^{\vee}(-1) \right)}{-1-\psi}.
\]

The vanishing of $J_{g_2,n}$ then gives us the relation
\begin{equation}\label{rel}
\hspace{-1cm}\frac{1}{2}\int_{\M_{g_2;n,0}(\B\Z_2)}\frac{e^{eq}\left( \mathbb{E}_{1}^{\vee}(-1) \right)}{-1-\psi}=-\sum_{h_1>0}\int_{\M_{h_1;0,1}(\B\Z_2)}\frac{e^{eq}\left( \mathbb{E}_{1}^{\vee}(1) \right)}{1-\psi}\cdot\int_{\M_{h_2;n,0}(\B\Z_2)}\frac{e^{eq}\left( \mathbb{E}_{1}^{\vee}(-1) \right)}{-1-\psi}.
\end{equation}

If $h_1>0$, then
\[
\int_{\M_{h_1;0,1}(\B\Z_2)}\frac{e^{eq}\left( \mathbb{E}_{1}^{\vee}(1) \right)}{1-\psi}=\int_{\M_{h_1;0,1}(\B\Z_2)}\sum (-1)^{g-i}\lambda_{g-i}\psi^{2g-2+i}
\]
where the $\lambda$ and $\psi$ classes are pulled back via the natural map $\rho:\M_{h_1;0,1}(\B\Z_2)\rightarrow \M_{h_1,1}$.  Therefore, the degree of the top intersection on the space of twisted curves is $\deg(\rho)$ times the degree of the corresponding top intersection on $\M_{h_1,1}$.  But we know that the corresponding integrals on $\M_{h_1,1}$ vanish because they are obtained by substituting $k=-1$ in the Faber-Pandharipande (\cite{fp:hiagwt}) formula:
\[
1+\sum_{i,g}t^{2g}k^i\int_{\M_{g,1}}\lambda_{g-i}\psi^{2g-2+i}=\left( \frac{t/2}{\sin(t/2)} \right)^{k+1}.
\]

Therefore, the right hand side of (\ref{rel}) vanishes.  The left hand side is a multiple of (\ref{int3}), completing the proof.

\end{proof}

\begin{lemma}\label{vanish3}
If a component contracted to $\infty$ has positive genus and the node attaching it to the teardrop is untwisted, then the contribution vanishes.
\end{lemma}

\begin{proof}
Consider a fixed locus with genus and mark splitting $g_1+g_2=g$ and $n_1+n_2=n-1$ where $n_2$ is even, forcing the node over $\infty$ to be untwisted.  The contribution to (\ref{int}) from such a fixed locus is
\begin{align*}
\frac{1}{4}\int_{\M_{g_1;n_1,0}(\B\Z_2)}&\frac{e^{eq}\left( \mathbb{E}_{1}^{\vee}(1)\oplus\mathbb{E}_{-1}^{\vee}(-1)\oplus\mathbb{E}_{-1}^{\vee}(0) \right)}{1-\psi}\\
&\cdot\int_{\M_{g_2;n_2,1}(\B\Z_2)}\frac{e^{eq}\left( \mathbb{E}_{1}^{\vee}(-1)\oplus\mathbb{E}_{-1}^{\vee}(-1/2)\oplus\mathbb{E}_{-1}^{\vee}(1/2) \right)}{-1-\psi}
\end{align*}

Again by the $\Z_2$ Mumford relation, the second integral simplifies to
\begin{equation}\label{int4}
\frac{(-1)^{g_2-1+n_2/2}}{2^{2g_2-2+n_2}}\int_{\M_{g_2;n_2,1}(\B\Z_2)}\frac{e^{eq}\left( \mathbb{E}_{1}^{\vee}(-1) \right)}{-1-\psi}
\end{equation}

We show that this integral vanishes when $g_2>0$.  If $g_2>0$, consider again the vanishing integral
\begin{equation}\label{int5}
J_{g_2,n_2+2}=\int_{\M_{g_2;n_2+1,0}(\mathcal{G},1)}\prod_{i=1}^{n_2+1}ev_i^*\so(1)
\end{equation}
where we now linearize the bundles in order to force $n_2$ points to $\infty$ and $1$ point to $0$.  The fixed loci are again indexed by genus splitting $h_1+h_2=g_2$.  The contribution from such a splitting is
\[
\int_{\M_{h_1;2,0}(\B\Z_2)}\frac{e^{eq}\left( \mathbb{E}_{1}^{\vee}(1) \right)}{1-\psi}\cdot\int_{\M_{h_2;n_2,1}(\B\Z_2)}\frac{e^{eq}\left( \mathbb{E}_{1}^{\vee}(-1) \right)}{-1-\psi}.
\]

The vanishing of $J_{g_2,n_2+2}$ gives us the relation
\begin{equation}\label{rel2}
\hspace{-1cm}\frac{1}{2}\int_{\M_{g_2;n_2,1}(\B\Z_2)}\frac{e^{eq}\left( \mathbb{E}_{1}^{\vee}(-1) \right)}{-1-\psi}=-\sum_{h_1>0}\int_{\M_{h_1;2,0}(\B\Z_2)}\frac{e^{eq}\left( \mathbb{E}_{1}^{\vee}(1) \right)}{1-\psi}\cdot\int_{\M_{h_2;n_2,1}(\B\Z_2)}\frac{e^{eq}\left( \mathbb{E}_{1}^{\vee}(-1) \right)}{-1-\psi}.
\end{equation}

The integral 
\[
\int_{\M_{h_1;2,0}(\B\Z_2)}\frac{e^{eq}\left( \mathbb{E}_{1}^{\vee}(1) \right)}{1-\psi}
\]
is seen to be a multiple of (\ref{int3}) (where $g_2=h_1$ and $n=2$), and we know from the proof of Lemma \ref{vanish2} that it vanishes if $h_1>0$.  Thus, the right side of (\ref{rel2}) vanishes.  The left side is a multiple of (\ref{int4}), completing the proof.


\end{proof}

\subsection{Proof of Formula}

We are now ready to prove formula (\ref{ident}).  By Lemmas \ref{vanish2} and \ref{vanish3}, loci contributing nontrivially to the integral map all genus to $0$.  By Lemma \ref{vanish1}, the only contributing locus with untwisted node over $0$ is the one for which all marks go to $\infty$ and the contracting genus at $0$ parametrizes the trivial double cover.  We now describe the surviving fixed loci.  For $0\leq k\leq\frac{n-2}{2}$, let $\Gamma_k$ denote the fixed locus of degree $1$ maps into $\G$ where the teardrop is twisted over $0$, the component contracted to $0$ has $n-1-2k$ twisted marked points and genus $g$, and the component over $\infty$ is rational and has $2k$ twisted marked points.  Let $\Gamma_{\frac{n}{2}}$ denote the locus of maps where the teardrop is twisted over $\infty$, all marked points are on a rational component contracted to $\infty$, and the component contracted to $\infty$ has genus $g$ and no marked points.  

The contribution from $\Gamma_0$ is seen to be $F_{g,n}$.  The contribution from $\Gamma_k$ with $1\leq k\leq \frac{n-2}{2}$ is computed to be
\begin{align*}
\frac{1}{4}{n-1 \choose 2k}\int_{\M_{0;2k,1}(\B\Z_2)}&\frac{e^{eq}\left(\mathbb{E}_{-1}^{\vee}(-1/2)\oplus\mathbb{E}_{-1}^{\vee}(1/2) \right)}{-1-\psi}\cdot F_{g,n-2k}\\
&=\frac{(-1)^{k}}{2^{2k}}{n-1 \choose 2k}\cdot F_{g,n-2k}
\end{align*}
where the equality follows from applying the $\Z_2$ Mumford relation and evaluating the top power of the $\psi$ class that remains (this is simply the Hurwitz number $\frac{1}{2}$).  The contribution from $\Gamma_{\frac{n}{2}}$ is computed to be
\[
-\int_{\mathcal{T}}\frac{e^{eq}\left( \mathbb{E}_{1}^{\vee}(1)\oplus\mathbb{E}_{-1}^{\vee}(-1)\oplus\mathbb{E}_{-1}^{\vee}(0) \right)}{1-\psi}\cdot\int_{\M_{0;n,0}(\B\Z_2)}\frac{e^{eq}\left(\mathbb{E}_{-1}^{\vee}(-1/2)\oplus\mathbb{E}_{-1}^{\vee}(1/2) \right)}{-1-\psi}.
\]

On the locus $\mathcal{T}$, we have $\mathbb{E}_1\cong\mathbb{E}_{-1}$.  Therefore, we can apply the Mumford relation and the contribution of $\Gamma_{\frac{n}{2}}$ simplifies to
\[
-\frac{1}{2}b_g\int_{\M_{0;n,0}(\B\Z_2)}\frac{e^{eq}\left(\mathbb{E}_{-1}^{\vee}(-1/2)\oplus\mathbb{E}_{-1}^{\vee}(1/2) \right)}{-1-\psi}=\frac{(-1)^{\frac{n}{2}}}{2^{n-1}}b_g
\]
where 
\[
b_g:=\int_{\M_{g,1}}\lambda_g\psi^{2g-2}.
\]

Summing these contributions to $0$, we get the recursion
\[
F_{g,n}=-\sum_{k=1}^{\frac{n-2}{2}}\frac{(-1)^{k}}{2^{2k}}{n-1 \choose 2k}\cdot F_{g,n-2k}-\frac{(-1)^{\frac{n}{2}}}{2^{n-1}}b_g.
\]
Inserting a formal variable $x$, we get
\[
F_{g,n}\frac{x^{n-1}}{(n-1)!}=-\sum_{k=1}^{\frac{n-2}{2}}\left[\frac{(-1)^{k}}{2^{2k}}\frac{x^{2k}}{(2k)!}\right]\cdot\left[F_{g,n-2k}\frac{x^{n-2k-1}}{(n-2k-1)!}\right]-\frac{(-1)^{\frac{n}{2}}}{2^{n-1}}\frac{x^{n-1}}{(n-1)!}b_g.
\]
Setting 
\[
F_g(x):=\sum_{n\geq 2}F_{g,n}\frac{x^{n-1}}{(n-1)!}
\]
we get
\[
F_g(x)=[-\cos\left(\frac{x}{2}\right)+1]F_g(x)+\sin\left(\frac{x}{2}\right)b_g.
\]
We solve $F_g(x)=\tan\left(\frac{x}{2}\right)b_g$.  The identity (\ref{ident}) now follows from the fact that
\[
\sum_{g\geq 0}b_g\lambda^{2g-1}=\frac{1}{2}\csc\left(\frac{\lambda}{2}\right).
\]

\bibliographystyle{alpha}
\bibliography{biblio}

\newcommand{\etalchar}[1]{$^{#1}$}
\begin{thebibliography}{MNOP06b}

\bibitem[AKMV05]{akmv:tv}
Mina Aganagic, Albrecht Klemm, Marcos Mari{\~n}o, and Cumrun Vafa.
\newblock The topological vertex.
\newblock {\em Comm. Math. Phys.}, 254(2):425--478, 2005.

\bibitem[BC11]{bc:ooinv}
Andrea Brini and Renzo Cavalieri.
\newblock Open orbifold gromov-witten invariants of {$\C^3/\Z_n$}: localization
  and mirror symmetry.
\newblock {\em Selecta Mathematica, New Series}, 17:879--933, 2011.
\newblock 10.1007/s00029-011-0060-4.

\bibitem[BCY10]{bcy:otv}
J.~Bryan, C.~Cadman, and B.~Young.
\newblock The orbifold topological vertex.
\newblock Preprint: math/1008.4205v1, 2010.

\bibitem[BG09]{bg:crc}
J.~Bryan and T.~Graber.
\newblock The crepant resolution conjecture.
\newblock {\em Proc. Sympos. Pure Math.}, 80:23--42, 2009.

\bibitem[BGP08]{bgp:crc}
J.~Bryan, T.~Graber, and R.~Pandharipande.
\newblock The orbifold quantum cohomology of {$\Bbb C\sp 2/Z\sb 3$} and
  {H}urwitz-{H}odge integrals.
\newblock {\em J. Algebraic Geom.}, 17(1):1--28, 2008.
\newblock arXiv:math.AG/0510335.

\bibitem[CC09]{cc:gl}
C.~Cadman and R.~Cavalieri.
\newblock Gerby localization, {$Z_3$}-{H}odge integrals and the {GW} theory of
  {$[\Bbb C^3/Z_3]$}.
\newblock {\em Amer. J. Math.}, 131(4):1009--1046, 2009.

\bibitem[CR04]{cr:nctoo}
W.~Chen and Y.~Ruan.
\newblock A new cohomology theory of orbifolds.
\newblock {\em Comm. Math. Phys.}, 248(1):1--31, 2004.

\bibitem[CR07]{cr:qcacr}
Tom Coates and Yongbin Ruan.
\newblock Quantum cohomology and crepant resolutions: A conjecture, 2007.

\bibitem[CR11]{cr:ogwcrc}
R.~Cavalieri and D.~Ross.
\newblock Open gromov-witten theory and the crepant resolution conjecture.
\newblock Preprint: math/1102.0717v1, 2011.

\bibitem[DF05]{df:lgta}
Duiliu-Emanuel Diaconescu and Bogdan Florea.
\newblock Localization and gluing of topological amplitudes.
\newblock {\em COMMUN.MATH.PHYS.}, 257:119, 2005.

\bibitem[FMN07]{fmn:stdms}
B.~Fantechi, E.~Mann, and F.~Nironi.
\newblock Smooth toric dm stacks.
\newblock {\em J. Reine Angew. Math.}, 2007.

\bibitem[FP00]{fp:hiagwt}
Carel Faber and Rahul Pandharipande.
\newblock {Hodge} integrals and {Gromov-Witten} theory.
\newblock {\em Invent. Math.}, 139(1):173--199, 2000.

\bibitem[GP99]{Graber-Pandharipande}
T.~Graber and R.~Pandharipande.
\newblock Localization of virtual classes.
\newblock {\em Invent. Math.}, 135(2):487--518, 1999.

\bibitem[HKK{\etalchar{+}}03]{clay:ms}
Kentaro Hori, Sheldon Katz, Albrecht Klemm, Rahul Pandharipande, Richard
  Thomas, Cumrun Vafa, Ravi Vakil, and Eric Zaslow.
\newblock {\em Mirror Symmetry}.
\newblock AMS CMI, 2003.

\bibitem[Joh09]{j:egwtods}
P.~Johnson.
\newblock Equivariant gromov-witten theory of one dimensional stacks.
\newblock Preprint: math/0903.1068v1, 2009.

\bibitem[JPT08]{jpt:ahhi}
P.~Johnson, R.~Pandharipande, and H.-H. Tseng.
\newblock Abelian hurwitz-hodge integrals.
\newblock Preprint: math/0803.0499v2, 2008.

\bibitem[KL02]{kl:oinv}
S.~Katz and M.~Liu.
\newblock Enumerative geometry of stable maps with lagrangian boundary
  conditions and multiple covers of the disk.
\newblock {\em Adv. Theor. Math. Phys.}, 5:1--49, 2002.

\bibitem[LLLZ09]{lllz:mttv}
Jun Li, Chiu-Chu~Melissa Liu, Kefeng Liu, and Jian Zhou.
\newblock A mathematical theory of the topological vertex.
\newblock {\em GEOM.TOPOL.}, 13:527, 2009.

\bibitem[Mac95]{m:sfhp}
I.~G. Macdonald.
\newblock {\em Symmetric functions and {H}all polynomials}.
\newblock Oxford Mathematical Monographs. The Clarendon Press Oxford University
  Press, New York, second edition, 1995.
\newblock With contributions by A. Zelevinsky, Oxford Science Publications.

\bibitem[MNOP06a]{mnop:gwdt1}
D.~Maulik, N.~Nekrasov, A.~Okounkov, and R.~Pandharipande.
\newblock Gromov-{W}itten theory and {D}onaldson-{T}homas theory. {I}.
\newblock {\em Compos. Math.}, 142(5):1263--1285, 2006.

\bibitem[MNOP06b]{mnop:gwdt2}
D.~Maulik, N.~Nekrasov, A.~Okounkov, and R.~Pandharipande.
\newblock Gromov-{W}itten theory and {D}onaldson-{T}homas theory. {II}.
\newblock {\em Compos. Math.}, 142(5):1286--1304, 2006.

\bibitem[MOOP11]{moop:gwdtc}
D.~Maulik, A.~Oblomkov, A.~Okounkov, and R.~Pandharipande.
\newblock Gromov-witten/donaldson-thomas correspondence for toric 3-folds.
\newblock {\em Inventiones Mathematicae}, 186:435--479, 2011.
\newblock 10.1007/s00222-011-0322-y.

\bibitem[ORV06]{orv:qcycc}
Andrei Okounkov, Nikolai Reshetikhin, and Cumrun Vafa.
\newblock Quantum {C}alabi-{Y}au and classical crystals.
\newblock In {\em The unity of mathematics}, volume 244 of {\em Progr. Math.},
  pages 597--618. Birkh\"auser Boston, Boston, MA, 2006.

\bibitem[Zon11]{z:gmvf}
Z.~Zong.
\newblock Generalized {M}ari{\~n}o-{V}afa formula and local {G}romov-{W}itten
  theory of orbi-curves.
\newblock Preprint: math/1109.4992v1, 2011.

\end{thebibliography}

\end{document}